\documentclass[12pt,reqno]{amsart}

\usepackage[a4paper,margin=1in]{geometry}
\usepackage{amsmath,amssymb,amsthm,mathtools,bm}
\usepackage{microtype}
\usepackage{enumitem}
\usepackage{booktabs}
\usepackage[hidelinks]{hyperref}
\usepackage[nameinlink,noabbrev]{cleveref}
\usepackage{bbm}
\usepackage{amsmath}
\usepackage{pgfplots}
\usepgfplotslibrary{groupplots}
\pgfplotsset{compat=1.18}

\newtheorem{theorem}{Theorem}[section]
\newtheorem{proposition}[theorem]{Proposition}
\newtheorem{lemma}[theorem]{Lemma}

\newtheorem{corollary}[theorem]{Corollary}
\newtheorem{remark}[theorem]{Remark}
\newtheorem{definition}[theorem]{Definition}
\newtheorem*{ac}{Acknowledgement}

\newcommand{\bC}{\mathbb{C}}
\newcommand{\bN}{\mathbb{N}}
\newcommand{\bR}{\mathbb{R}}
\newcommand{\cL}{\mathcal{L}}
\newcommand{\cM}{\mathcal{M}}
\newcommand{\cH}{\mathcal{H}}
\newcommand{\cS}{\mathcal{S}}
\newcommand{\cG}{\mathcal{G}}

\newcommand{\cT}{\mathcal{T}}
\newcommand{\cV}{\mathcal{V}}
\newcommand{\bE}{\mathbb{E}}
\newcommand{\cI}{\mathcal{I}}
\newcommand{\Tr}{\mathrm{Tr}}
\newcommand{\Id}{\mathrm{Id}}
\newcommand{\K}{\mathbf{K}}
\newcommand{\bO}{\mathbf{O}}

\newcommand{\GNS}{\mathrm{GNS}}
\newcommand{\BKM}{\mathrm{BKM}}
\DeclareMathOperator{\spec}{spec}
\DeclareMathOperator{\spanop}{span}
\newcommand{\diag}{\mathrm{diag}}
\newcommand{\ip}[2]{\langle #1,#2\rangle}
\allowdisplaybreaks[2]
\renewcommand{\Re}{\operatorname{Re}}

\title[ENTROPIC INEQUALITIES FOR PRIMITIVE KMS SYMMETRIC QMS]{SOME ENTROPIC INEQUALITIES FOR PRIMITIVE KMS SYMMETRIC QUANTUM MARKOV SEMIGROUPS}

\author{Zhengwei Liu}
\address{Zhengwei Liu, Tsinghua University, Beijing}
\email{liuzhengwei@tsinghua.edu.cn }

\author{Jincheng Wan}
\address{Jincheng Wan, Tsinghua University, Beijing}
\email{wanjc23@mails.tsinghua.edu.cn }

\author{Jinsong Wu}
\address{Jinsong Wu, Beijing Institute of Mathematical Sciences and Applications, Beijing, 101408, China}
\email{wjs@bimsa.cn}

\date{}

\begin{document}

\maketitle

\begin{abstract}
In this note, we prove that every primitive KMS-symmetric quantum Markov
semigroup on a finite-dimensional matrix algebra satisfies a modified
logarithmic Sobolev inequality (MLSI).
We also construct a primitive quantum Markov semigroup without KMS symmetry
that fails MLSI, and a primitive KMS-symmetric quantum Markov semigroup
that fails complete MLSI (CMLSI).
The latter construction also provides a non-primitive KMS-symmetric
quantum Markov semigroup without MLSI.
For the graph-based KMS-symmetric quantum Markov semigroups studied here, we prove MLSI and CMLSI when the underlying graph is connected and has at least three vertices.
Finally, we establish CMLSI for a class of primitive bimodule KMS-symmetric quantum Markov semigroups arising from fermionic systems.
\end{abstract}

\section{Introduction}

Quantum Markov semigroups constitute a fundamental mathematical framework for the study of open quantum systems \cite{Dav74}.
Gorini, Kossakowski, Lindblad, and Sudarshan \cite{GKS76,Lin76} laid the foundations of the theory of quantum Markov semigroups and their generators.
A quantum Markov semigroup is said to be in equilibrium if it is invariant under a faithful normal state.
Frigerio and Verri \cite{Fri78,FriVer82} characterized the fixed-point space and the multiplicative domain in the equilibrium setting.
In the noncommutative framework, an equilibrium state gives rise to various inner products, including the GNS (Gelfand–Naimark–Segal), KMS (Kubo–Martin–Schwinger), and BKM (Bogoliubov–Kubo–Mori) inner products.
A quantum Markov semigroup that is self-adjoint with respect to the GNS (respectively, KMS) inner product is called GNS symmetric (respectively, KMS symmetric).
When the symmetry does not originate from the equilibrium state, the notions of bimodule GNS symmetry and bimodule KMS symmetry were introduced by Zhao, Jiang, Wan, and Wu \cite{WuZha25,JWW25,JWW26}.
In this paper, we focus primarily on KMS symmetric quantum Markov semigroups.
These semigroups have been studied intensively by Kossakowski, Frigerio, Gorini, and Verri \cite{KFGV77}, Fagnola and Umanità \cite{FagUma10}, Vernooij and Wirth \cite{VerWir23}, Wirth \cite{Wir26}, among others.

A central objective in the theory of KMS symmetric quantum Markov semigroups is to establish modified logarithmic Sobolev inequalities (MLSI) and complete modified logarithmic Sobolev inequalities (CMLSI) \cite{BGJ22,BGJ23}.
In 2017, Carlen and Maas obtained MLSI for certain GNS symmetric quantum Markov semigroups via the intertwining property \cite{CarMaa17}.
In 2022, Gao and Rouzé proved that every finite-dimensional GNS-symmetric quantum Markov semigroup admits a CMLSI.
In 2025, Gao, Junge, LaRacuente, and Li \cite{GJLL25} showed that, for GNS symmetric quantum Markov semigroups, the CMLSI constant is bounded by the inverse of the complete positivity mixing time.
Problem 6.3 in \cite{GJLL25} asks whether their entropy decay and entropy difference results extend to the KMS-symmetric setting.

In this paper, we prove that every primitive KMS-symmetric quantum Markov semigroup on a finite-dimensional matrix algebra admits an MLSI.
Nevertheless, we construct a counterexample of a primitive KMS symmetric quantum Markov semigroup that fails the CMLSI.
Consequently, not all KMS symmetric quantum Markov semigroups admit a CMLSI.
This partially answers the question of Gao, Junge, LaRacuente, and Li in \cite{GJLL25}.
We also study graph-based KMS-symmetric quantum Markov semigroups and a class
of primitive bimodule KMS-symmetric quantum Markov semigroups arising from
fermionic systems, and establish CMLSI under the hypotheses stated below.

The paper is organized as follows.
In Section 2, we shall recall some basic notations for KMS-symmetric quantum Markov semigroups including Fisher information and relative entropy.
In Section 3, we shall prove that any primitive KMS-symmetric quantum Markov semigroup has a modified logarithmic Sobolev inequality.
In Section 4, we present an example of primitive  quantum Markov semigroup without MLSI.
In Section 5, we present an example of primitive KMS-symmetric quantum Markov semigroup without CMLSI.
In Section 6, we obtain the CMLSI for graph-based KMS-symmetric quantum Markov semigroups.
In Section 7, we obtain the CMLSI for bimodule KMS-symmetric quantum Markov semigroups.

\begin{ac}
The paper is based on the numerous communications with ChatGPT Pro 5.6 and 6.
Zhengwei Liu was supported by Beijing Natural Science Foundation Key Program (Grant No. Z220002) and Beijing Natural Science Foundation (Grant No. Z221100002722017).
J. Wu was supported by grants from Beijing Institute of Mathematical Sciences and Applications.
J.~W. was supported by NSFC (Grant no. 12371124). 
\end{ac}

\section{Preliminaries}

In this section, we shall briefly recall the notations for KMS-symmetric quantum Markov semigroups on matrix algebras.

Suppose that $\{\Phi_t\}_{t\geq 0}$ is a quantum Markov semigroup on the matrix algebra $M_d(\bC)$ with unnormalized trace $\Tr$.
Let $\cL$ be the Lindbladian of $\{\Phi_t\}_{t\geq 0}$ such that $\Phi_t=e^{-t\cL}$.
The GKSL form of the generator is as follows:
\begin{align}\label{eq:gksl}
\cL(X)=i[H, X] +\sum_{j=1}^m \frac{1}{2}\{V_j^*V_j, X\}-\sum_{j=1}^m V_j^*XV_j,
\end{align}
where $H\in M_d(\bC)$ is hermitian and $V_j\in M_d(\bC)$.
A quantum Markov semigroup $\{\Phi_t\}_{t\geq 0}$ is called
primitive if it admits a faithful invariant state and
$\ker(\cL)=\bC I$. Equivalently, it has a unique invariant state
$\rho$, which is faithful and satisfies $\rho\Phi_t=\rho$
for all $t\geq 0$.

A quantum Markov semigroup $\{\Phi_t\}_{t\geq 0}$ is KMS-symmetric with respect to a faithful state $\rho$ if for all $t\geq 0$ and $X,Y\in M_d(\bC)$,
\begin{align}
    \Tr(X^* D_{\rho}^{1/2}\Phi_t(Y)D_\rho^{1/2})=\Tr(\Phi_t(X)^* D_{\rho}^{1/2}YD_\rho^{1/2}),
\end{align}
where $D_\rho$ is the density matrix of $\rho$, i.e. $\rho(\cdot)=\Tr(D_\rho \cdot)$.
This implies that $\{\Phi_t\}_{t\geq 0}$ is in equilibrium with respect to $\rho$.
For a quantum Markov semigroup admitting a faithful invariant state,
we denote by $\bE_\Phi$ its ergodic projection onto $\ker\cL$:
\[
\bE_\Phi(X)
=\lim_{T\to\infty}\frac1T\int_0^T\Phi_t(X)\,dt,
\qquad X\in M_d(\bC).
\]
In finite dimensions, this limit exists in norm.
If the semigroup is KMS-symmetric, then
\[
\bE_\Phi=\lim_{t\to\infty}\Phi_t.
\]

The Fisher information $\cI_{\cL}$ of the semigroup is
\begin{align*}
\cI_{\cL}(D)=\Tr(\cL^*(D)(\log D-\log \bE_{\Phi}^*(D))), 
\end{align*}
where $D$ is a faithful density matrix in $M_d(\bC)$ and $\cL^*$ is the adjoint of $\cL$ with respect to $\Tr$.
When $\{\Phi_t\}_{t\geq 0}$ is primitive, we have that $\cI_{\cL}(D)=\Tr(\cL^*(D)(\log D-\log D_\rho))$.
The relative entropy $H(D\| D_\rho)$ is $\Tr(D(\log D-\log D_\rho))$.

A quantum Markov semigroup $\{\Phi_t\}_{t\geq 0}$ admitting a faithful
invariant state, with ergodic projection $\bE_\Phi$, satisfies a modified
logarithmic Sobolev inequality if there exists $\beta>0$ such that,
for every density matrix $D$ and every $t\geq0$,
\begin{align}\label{eq:mlsi}
H(\Phi_t^*(D) \| \bE_{\Phi}^*(D)) \leq e^{-2\beta t} H(D \| \bE_{\Phi}^*(D)),
\end{align}
equivalently, for every faithful density matrix $D$,
\begin{align}\label{eq:mlsi2}
\cI_{\cL}(D) \geq 2\beta H(D \| \bE_{\Phi}^*(D)).
\end{align}
In the primitive case, inequalities \eqref{eq:mlsi} and \eqref{eq:mlsi2} reduce to
\[
H(\Phi_t^*(D)\|D_\rho)\leq e^{-2\beta t}H(D\|D_\rho),
\qquad
\cI_{\cL}(D)\geq 2\beta H(D\|D_\rho).
\]
We denote by $\alpha_{MLSI}$ the optimal constant
in inequalities \eqref{eq:mlsi} and \eqref{eq:mlsi2}.

In the following, we recall the complete modified logarithmic Sobolev inequalities for quantum Markov semigroups.
\begin{definition}[Complete Modified Logarithmic Sobolev Inequalities]
Let $\{\Phi_t\}_{t\geq0}$ be a quantum Markov semigroup on $M_d(\bC)$
admitting a faithful invariant state, and let $\bE_\Phi$ be its ergodic
projection. For $\beta\geq0$, the semigroup satisfies
$\operatorname{CMLSI}(\beta)$ if, for every $n\in\bN$ and every faithful
density matrix $D$ on $M_d\otimes M_n$,
\begin{equation}\label{eq:cmlsi}
\cI_{\cL\otimes\Id_n}(D)\geq 2\beta H(D\|(\bE_\Phi^* \otimes \Id_n )(D)).
\end{equation}
If the semigroup is primitive with invariant density matrix $D_\rho$,
then $\bE_\Phi^*(X)=\Tr(X)D_\rho$. Consequently,
\[
(\bE_\Phi^*\otimes\Id_n)(D)=D_\rho\otimes D_R,
\qquad D_R=\operatorname{Tr}_{M_d}(D).
\]
Complete entropic inequalities in this form are developed in \cite{GaoRou22}.
\end{definition}
We denote by $\alpha_{\mathrm{CMLSI}}(\cL)$ the supremum of all
$\beta\geq0$ for which \eqref{eq:cmlsi} holds for every $n\in\bN$
and every faithful density matrix $D\in M_d\otimes M_n$.
When the generator is clear from the context, we omit the argument $\cL$.

The logarithmic-mean operator $\K_{D}$ on $M_d(\bC)$ is defined as follows:
\begin{align*}
\K_D(X)=\int_0^1 D^s X D^{1-s} ds, \quad X\in M_d(\bC).
\end{align*}
For any $X\in M_d(\bC)$ and $s\in [0,1]$, the BKM-norm of $X$ with respect to a faithful state $\rho$ is defined to be
\begin{align*}
\|X\|_{\BKM,\rho}=\left(\int_0^1 \Tr(D_\rho^{s} X^* D_\rho^{1-s} X) ds\right)^{1/2}.
\end{align*}
The 2-norm of $X$ with respect to a faithful state $\rho$ at $s$ is defined to be
\begin{align*}
  \|X\|_{s, \rho}=\left( \Tr(D_\rho^{s} X^* D_\rho^{1-s} X) \right)^{1/2}.
\end{align*}
For $s=1$, this is the GNS norm
\[
\|X\|_{\GNS,\rho}
=\bigl(\Tr(D_\rho X^*X)\bigr)^{1/2}.
\]
For $s=0$, it equals $\|X^*\|_{\GNS,\rho}$,
whereas for $s=1/2$ it is the KMS norm.
The BKM-inner product $\langle \cdot, \cdot \rangle_{\BKM, \rho}$ is defined to be
\begin{align*}
    \langle X, Y\rangle_{\BKM, \rho} =\int_0^1 \Tr(D_\rho^{s} X^* D_\rho^{1-s} Y)ds, \quad X, Y \in M_d(\bC).
\end{align*}
For an operator $\cL$ on $M_d(\bC)$, the norm $\|\cL\|_{\BKM, \rho}$ is defined to be $\displaystyle \sup_{X\neq 0} \frac{\|\cL(X)\|_{\BKM, \rho}}{\|X\|_{\BKM, \rho}}$.
In what follows, $\|\cdot\|_\rho$ and $\langle\cdot,\cdot\rangle_\rho$
denote the BKM norm and BKM inner product with respect to $\rho$, respectively.

\begin{lemma}\label{lem:tech1}
Suppose that $\{\Phi_t\}_{t\geq 0}$ is a quantum Markov semigroup in equilibrium with respect to a faithful state $\rho$.
Then $\|\Phi_t(X)\|_{s, \rho} \leq \|X\|_{s, \rho}$ and $\Re \Tr(D_\rho^s X^* D_\rho^{1-s}\cL(X)) \geq 0$ for all $s\in [0,1]$.
\end{lemma}
\begin{proof}
By the Kadison-Schwarz inequality, we have that 
\begin{align*}
\Tr(D_\rho \Phi_t(X)^*\Phi_t(X)) \leq \Tr(D_\rho \Phi_t(X^*X)), \quad \Tr(D_\rho \Phi_t(X)\Phi_t(X)^*) \leq \Tr(D_\rho \Phi_t(XX^*)),
\end{align*}
i.e. $\|\Phi_t(X)\|_{1, \rho}\leq \|X\|_{1, \rho}$ and  $\|\Phi_t(X)\|_{0, \rho}\leq \|X\|_{0, \rho}$.
By \cite{Kos84}, we have that $\|\Phi_t(X)\|_{s, \rho} \leq \|X\|_{s, \rho}$ for $s\in [0,1]$.
Differentiating at $t=0$ gives
\[
0\geq
\left.\frac{d}{dt}\|\Phi_t(X)\|_{s,\rho}^2\right|_{t=0}
=-2\Re\Tr\!\left(D_\rho^sX^*D_\rho^{1-s}\cL(X)\right),
\]
which proves the second assertion.
\end{proof}

\begin{lemma}\label{lem:tech2}
For sufficiently small $\varepsilon>0$, let $D_{\varepsilon}=D_\rho + \varepsilon \K_{D_\rho} (X)$, where $X=X^*$, $\|X\|\leq 1$ and $\rho(X)=0$.
We have that 
\begin{align}
H(D_{\varepsilon} \| D_\rho)=& \frac{\varepsilon^2}{2} \int_0^1 \Tr(D_\rho^s X D_\rho^{1-s} X) ds +\bO(\varepsilon^3), \label{eq:entropyest}\\
\cI_{\cL}(D_\varepsilon) = & \varepsilon^2 \int_0^1 \Re \Tr(D_\rho^s XD_\rho^{1-s}\cL(X)) ds +\bO(\varepsilon^3). \label{eq:informationest} 
\end{align}
\end{lemma}
\begin{proof}
Let $f(\varepsilon)=H(D_{\varepsilon} \| D_\rho)$.
We have that 
\begin{align*}
f(\varepsilon) =& \Tr(D_\varepsilon (\log D_{\varepsilon}-\log D_\rho)) \\
=& \Tr(D_\rho (\log D_{\varepsilon}-\log D_\rho)) +\varepsilon \Tr(\K_{D_\rho} (X) (\log D_{\varepsilon}-\log D_\rho)).
\end{align*}
By Example 3.22 in \cite{HiaPet14}, we have that 
\begin{align*}
\log D_{\varepsilon}-\log D_\rho= \varepsilon \K_{D_\rho}^{-1} \K_{D_\rho} (X) +\bO(\varepsilon^2) =  \varepsilon  X+ \bO(\varepsilon^2).
\end{align*}
By taking the differentiation of $f(\varepsilon)$ with respect to $\varepsilon$ and $\Tr(D_\rho X)=0$, we obtain that 
\begin{align*}
f'(\varepsilon)=& \Tr(D_\rho X)+ \Tr(\K_{D_\rho} (X) (\log D_{\varepsilon}-\log D_\rho)) \\
=&  \Tr(\K_{D_\rho} (X) (\log D_{\varepsilon}-\log D_\rho)).
\end{align*}
Hence $f'(0)=0$.
By taking the  differentiation of $f'(\varepsilon)$ with respect to $\varepsilon$, we have that 
\begin{align*}
f''(0) =& \Tr(\K_{D_\rho} (X) \K_{D_\rho}^{-1}(\K_{D_\rho}(X)))= \Tr(\K_{D_\rho} (X) X) \\
=&\int_0^1 \Tr(D_\rho^s X D_\rho^{1-s}X) ds.
\end{align*}
By taking the Taylor expansion with respect to $\varepsilon$ and $f(0)=0$, we obtain that 
\begin{align*}
H(D_{\varepsilon} \| D_\rho) = \frac{\varepsilon^2}{2}\int_0^1 \Tr(D_\rho^s X D_\rho^{1-s}X) \; ds+\bO(\varepsilon^3).
\end{align*}
By the fact that $\cL^*(D_\rho)=0$, we have that 
\begin{align*}
\cI_{\cL} (D_{\varepsilon}) =&\Tr(\cL^*(D_{\varepsilon})(\log D_\varepsilon -\log D_\rho) ) \\
=& \varepsilon \Tr(\cL^*(\K_{D_\rho}(X))(\log D_\varepsilon -\log D_\rho) )\\
=& \varepsilon^2 \Tr(\cL^*(\K_{D_\rho}(X))X) +\bO(\varepsilon^3)\\
=&  \varepsilon^2 \int_0^1 \Tr( D_\rho^s X D_\rho^{1-s}\cL(X)) ds+\bO(\varepsilon^3).
\end{align*}
This completes the proof of the lemma.
\end{proof}

\section{MLSI for Primitive KMS Symmetric QMS}
In this section, we shall obtain the modified logarithmic Sobolev inequality for primitive KMS symmetric quantum Markov semigroups and analyze their structures.
\begin{theorem}\label{thm:mlsiprim}
Suppose that $\{\Phi_t\}_{t \geq 0}$ is a primitive KMS-symmetric quantum Markov semigroup on $M_d(\bC)$ in equilibrium with respect to a faithful state $\rho$.
Then there exists $\beta>0$ such that the modified log-Sobolev inequality \eqref{eq:mlsi} holds.
\end{theorem}
\begin{proof}
The case $d=1$ is trivial, so we assume $d\geq2$.
We first prove the Fisher information inequality for faithful density matrices.
Let 
\begin{align}\label{eq:keybd}
Q(X)=\frac{\displaystyle \int_0^1 \Re \Tr(D_\rho^s X D_\rho^{1-s}\cL(X))ds}{\displaystyle \int_0^1 \Tr(D_\rho^s X D_\rho^{1-s}X) ds}, 
\end{align}
where 
\begin{align*}
X\in  \cS_\rho=\{X=X^*\in M_d(\bC): \Tr(D_\rho X)=0, \|X\|=1\}.
\end{align*}
Note that $\displaystyle \int_0^1 \Tr(D_\rho^s X D_\rho^{1-s}X) ds>0$ for $X\neq 0$.
This indicates that $Q(X)$ is well-defined.
By Lemma \ref{lem:tech1}, for all $s\in [0,1]$ we have
$\Re \Tr(D_\rho^s X D_\rho^{1-s}\cL(X))\geq 0$.
By KMS symmetry and primitivity, we have
\[
\Re\Tr\!\left(D_\rho^{1/2}XD_\rho^{1/2}\cL(X)\right)>0,
\qquad X\in\cS_\rho.
\]
Hence $Q(X)>0$.
Now by the fact that $\cS_\rho$ is compact in $M_d(\bC)$, we see that $\displaystyle \inf_{X\in \cS_\rho}Q(X)>0$.

Suppose that $D\leq \kappa D_\rho$ for some $\kappa\geq 2$ for every density matrix $D$ in $M_d(\bC)$.
In general, one could take $\kappa$ to be the inverse of the minimal eigenvalue of $D_\rho$.
Let $q=\displaystyle \inf_{X\in \cS_\rho}Q(X)$ and $p=\|\cL\|_{\BKM, \rho}$, $\displaystyle \varepsilon=\frac{q}{3p}\leq \frac{1}{3}$.
Suppose that 
\begin{align*}
    - \varepsilon D_\rho \leq D-D_\rho \leq \varepsilon D_\rho.
\end{align*}
Hence for any $0\leq t\leq 1$, by the joint monotonicity of the logarithmic Kubo-Ando mean, we have that 
\begin{align*}
\frac{1}{1+t\varepsilon}\K_{D_\rho}^{-1} \leq \K_{(1-t)D_\rho +tD}^{-1}\leq \frac{1}{1-t\varepsilon}\K_{D_\rho}^{-1}.
\end{align*}
By noting that 
\begin{align*}
    \log D-\log D_\rho =\int_0^1 \K_{(1-t)D_\rho +tD}^{-1} (D-D_\rho) dt,
\end{align*}
we obtain that 
\begin{align*}
  &  \left\|\log D-\log D_\rho -\K_{D_\rho}^{-1}(D-D_\rho)\right\|_\rho  \\
   \leq & \int_0^1 \|(\K_{(1-t)D_\rho +tD}^{-1}-\K_{D_\rho}^{-1}) (D-D_\rho)\|_\rho dt\\
   =& \int_0^1 \| (\K_{D_\rho}^{1/2}\K_{(1-t)D_\rho +tD}^{-1} \K_{D_\rho}^{1/2}-\Id)\K_{D_\rho}^{-1/2} (D-D_\rho)\|_2 dt \\
   \leq & \int_0^1 \frac{t\varepsilon}{1-t\varepsilon} dt\|\K_{D_\rho}^{-1}(D-D_\rho)\|_\rho\\
    \leq & \frac{\varepsilon}{2(1-\varepsilon)}\|\K_{D_\rho}^{-1}(D-D_\rho)\|_\rho \\
    \leq&  \frac{q}{4p}\|\K_{D_\rho}^{-1}(D-D_\rho)\|_\rho,
\end{align*}
i.e.
\begin{align}\label{eq:normest1}
\left\|\log D-\log D_\rho -\K_{D_\rho}^{-1}(D-D_\rho)\right\|_\rho \leq  \frac{q}{4p}\|\K_{D_\rho}^{-1}(D-D_\rho)\|_\rho.
\end{align}
Then by Equation \eqref{eq:normest1}, we have that 
\begin{equation}\label{eq:inforest}
\begin{aligned}
    \cI_{\cL}(D) =& \langle \K_{D_\rho}^{-1}(D-D_\rho), \cL(\log D-\log D_\rho)\rangle_\rho \\
    =& \langle \K_{D_\rho}^{-1}(D-D_\rho), \cL(\K_{D_\rho}^{-1}(D-D_\rho))\rangle_\rho  \\
    & +\langle \K_{D_\rho}^{-1}(D-D_\rho), \cL(\log D-\log D_\rho-\K_{D_\rho}^{-1}(D-D_\rho))\rangle_\rho\\
    \geq & q\|\K_{D_\rho}^{-1}(D-D_\rho)\|_\rho^2 - \frac{p\varepsilon }{2(1-\varepsilon)} \|\K_{D_\rho}^{-1}(D-D_\rho)\|_\rho^2\\
    \geq &  \frac{3q}{4}\|\K_{D_\rho}^{-1}(D-D_\rho)\|_\rho^2.
\end{aligned}
\end{equation}

On the other hand, the integral form of Taylor's theorem gives
\begin{equation}\label{eq:entropy-local-bound}
\begin{aligned}
H(D\|D_\rho)
&=
\int_0^1 (1-t)\,
\operatorname{Tr}\!\left(
(D-D_\rho)
K_{(1-t)D_\rho+tD}^{-1}(D-D_\rho)
\right)\,dt
\\
&\le
\left(\int_0^1 \frac{1-t}{1-t\varepsilon}\,dt\right)
\operatorname{Tr}\!\left(
(D-D_\rho)K_{D_\rho}^{-1}(D-D_\rho)
\right)
\\
&\le
\frac{1}{2(1-\varepsilon)}
\left\|K_{D_\rho}^{-1}(D-D_\rho)\right\|_\rho^2
\\
&\le
\frac34
\left\|K_{D_\rho}^{-1}(D-D_\rho)\right\|_\rho^2.
\end{aligned}
\end{equation}
Equations \eqref{eq:inforest} and \eqref{eq:entropy-local-bound} indicate that 
\begin{align}\label{eq:mslest1}
    \cI_{\cL}(D)\geq q H(D\|D_\rho)
\end{align}
for $-\varepsilon D_\rho\leq D-D_\rho \leq \varepsilon D_\rho$.


Let $\displaystyle t_0=\frac{\varepsilon}{\kappa-1}$ and $D_{\rho_0}=(1-t_0)D_\rho +t_0 D$.
Then $-\varepsilon D_\rho \leq D_{\rho_0}-D_\rho \leq \varepsilon D_\rho$.
By the convexity of the Fisher information \cite{Spo78}, we have that 
\begin{align*}
\cI_{\cL}(D_{\rho_0}) \leq (1-t_0) \cI_{\cL}(D_\rho) +t_0 \cI_{\cL}(D).
\end{align*}
By Inequality \eqref{eq:mslest1}, we have that 
\begin{align*}
\cI_{\cL}(D) \geq t_0^{-1} \cI_{\cL}(D_{\rho_0}) \geq q t_0^{-1} H(D_{\rho_0} \| D_\rho).
\end{align*}
By Lemma 2.2 in \cite{GaoRou22}, we have that
\begin{align*}
    H(D_{\rho_0}\|D_\rho) \geq \frac{(1+\varepsilon)\log (1+\varepsilon)-\varepsilon}{\varepsilon^2} t_0^2 H(D\|D_\rho)
\end{align*}
and 
\begin{align*}
\frac{(1+\varepsilon)\log (1+\varepsilon)-\varepsilon}{\varepsilon^2} \geq \frac{1}{2+\varepsilon}\geq \frac{3}{7}.
\end{align*}
We have that 
\begin{align}\label{eq:mlsiest0}
\cI_{\cL}(D) \geq   \frac{\left(\inf_{X\in \cS_\rho} Q(X)\right)^2}{7\|\cL\|_{\BKM, \rho}(\kappa-1)}  H(D \| D_\rho).
\end{align}
In this case, we have that 
\begin{align*}
 \alpha_{MLSI} \geq \frac{\left(\inf_{X\in \cS_\rho} Q(X)\right)^2}{14\|\cL\|_{\BKM, \rho}(\kappa-1)}.
 \end{align*}
The entropy contraction inequality for an arbitrary density matrix $D$
follows by applying it to
\[
D_\delta=(1-\delta)D+\delta D_\rho,\qquad 0<\delta<1,
\]
and letting $\delta\downarrow0$.
This completes the proof of the theorem.
\end{proof}

In the following, we rewrite the value $\displaystyle \inf_{X\in \cS_\rho} Q(X)$.
We assume that $\{E_{jk}\}_{j,k=1}^d$ is a system of matrix units of $M_d(\bC)$ and $\displaystyle D_\rho=\sum_{j=1}^d s_j E_{jj}$.
Then 
\begin{align}\label{eq:kd}
    \K_{D_\rho}(E_{jk}) = t_{jk}E_{jk}, \quad t_{jk}
    =\left\{\begin{array}{ll} \displaystyle \frac{s_j-s_k}{\log s_j -\log s_k} & s_j\neq s_k\\ s_j & s_j=s_k \end{array}\right..
\end{align}
Use column vectorization
$\operatorname{vec}(X)=\sum_{k=1}^d e_k\otimes Xe_k$,
where $E_{jk}=e_je_k^*$, and let $L$ satisfy
$\operatorname{vec}(\cL(X))=L\operatorname{vec}(X)$.
Since $t_{jk}=t_{kj}$, the diagonal matrix
\[
W_\rho=\sum_{j,k=1}^d t_{jk}E_{jj}\otimes E_{kk}>0
\]
satisfies
$\operatorname{vec}(\K_{D_\rho}(X))
=W_\rho\operatorname{vec}(X)$.
By homogeneity, we obtain
\[
\inf_{X\in\cS_\rho}Q(X)
=
\min_{\substack{X=X^*\neq0\\
                 \operatorname{Tr}(D_\rho X)=0}}
\frac{
\operatorname{Re}\bigl(
\operatorname{vec}(X)^*W_\rho L\operatorname{vec}(X)
\bigr)}
{\operatorname{vec}(X)^*W_\rho\operatorname{vec}(X)}.
\]
This is the minimum generalized Rayleigh quotient on the
real linear space of self-adjoint matrices with zero $D_\rho$-mean.

\section{A Primitive QMS without MLSI}

In this section, we shall show that primitivity alone does not imply a modified logarithmic Sobolev inequality once KMS symmetry is removed. The
counterexample already occurs on $M_2(\mathbb{C})$.

Suppose that $M_2(\bC)$ is generated by unitary elements $P, Q$ in $M_2(\bC)$ subject to
\begin{align*}
PQ+QP=0, \quad P^2=Q^2=I.
\end{align*}
Let $W=-iPQ$.
We shall consider the semigroup arising from the following generator $\cL$ on $M_2(\bC)$:
\begin{align}\label{eq:m2lind1}
    \cL(X)=-i\omega[P,X]+\gamma(X-WXW),
\end{align}
where $\omega, \gamma>0$.
Let $\Phi_t=e^{-t\cL}$.
\begin{proposition}
We have that  the semigroup $\{\Phi_t\}_{t\geq 0}$ given by \eqref{eq:m2lind1} is a primitive quantum Markov semigroup.
\end{proposition}
\begin{proof}
Note that $\cL$ satisfies the GKSL form.
We see that $\{\Phi_t\}_{t\geq 0}$ is a quantum Markov semigroup.
Note that 
\begin{align*}
\cL^*(X) =i\omega [P, X]+\gamma (X-WXW).
\end{align*}
We see that $\cL^*(I)=0$ and $\displaystyle \frac{1}{2}\Tr$ is an equilibrium state.

Observing that 
\begin{align*}
\cL(I)=0, \quad \cL(P)=2\gamma P, \quad \cL(Q)=2\omega W+2\gamma Q, \quad \cL(W)=-2\omega Q,
\end{align*}
we see that the eigenvalues of $\cL$ are $0$, $2\gamma$, $\gamma\pm \sqrt{\gamma^2-4\omega^2}$.
Note that all nonzero eigenvalues have strictly positive real parts.
We obtain that $\displaystyle \frac{1}{2}\Tr$ is the unique equilibrium state.

If $\{\Phi_t\}_{t\geq0}$ were KMS-symmetric, uniqueness of the
invariant state would force it to be symmetric with respect to
$\frac12\Tr$. This would imply $\cL=\cL^*$.
However,
\[
\cL-\cL^*=-2i\omega[P,\cdot]\neq0,
\]
so the semigroup is not KMS-symmetric.
\end{proof}

\begin{proposition}
The best constant in inequality \eqref{eq:mlsi} for the semigroup $\{\Phi_t\}_{t \geq 0}$ with generator \eqref{eq:m2lind1} is zero.
This indicates that $\{\Phi_t\}_{t \geq 0}$ does not have a modified log Sobolev inequality in the form \eqref{eq:mlsi}.
\end{proposition}
\begin{proof}
Suppose that $0<\varepsilon<1$ and $D_\varepsilon=\frac12(I+\varepsilon W)$.
We shall evaluate $\cI_{\cL}(D_\varepsilon)$.
Note that 
\begin{align}
    \log D_{\varepsilon}-\log(I/2)
    =\frac{1}{2}\log(1-\varepsilon^2)\,I+\operatorname{arctanh}(\varepsilon)W.
    \label{eq:non-kms-log-witness}
\end{align}
and $\cL^*(D_\varepsilon)=\omega \varepsilon Q$.
We have that 
\begin{equation}
    \begin{aligned}
    \cI_{\cL}(D_\varepsilon)
    &=\Tr\!\left(\cL^*(D_\varepsilon)
      (\log D_\varepsilon-\log(I/2))\right)\\
    &=\frac{\omega\varepsilon}{2}\log(1-\varepsilon^2)\Tr(Q)
      +\omega\varepsilon\operatorname{arctanh}(\varepsilon)\Tr(QW)
      =0.
    \end{aligned}
    \label{eq:non-kms-zero-production}
\end{equation}
Nevertheless,
\begin{equation}
    H(D_\varepsilon \| I/2)
    =\frac{1+\varepsilon}{2}\log(1+\varepsilon)
     +\frac{1-\varepsilon}{2}\log(1-\varepsilon)>0.
    \label{eq:non-kms-positive-entropy}
\end{equation}
Hence there is no positive value $\beta$ such that Equation \eqref{eq:mlsi} is true.
\end{proof}

\section{A Primitive KMS-Symmetric QMS without CMLSI}

In this section, we construct a primitive KMS-symmetric quantum Markov semigroup on $M_2(\bC)$ that does not satisfy CMLSI.

Let
\begin{equation}\label{eq:qubitdata}
 V=\begin{pmatrix}0&1\\2&0\end{pmatrix},
 \qquad
D_\rho=\frac15\begin{pmatrix}1&0\\0&4\end{pmatrix},
 \qquad
D_\rho^{1/2}=\frac1{\sqrt5}\begin{pmatrix}1&0\\0&2\end{pmatrix}.
\end{equation}
Define the Lindbladian generator $\cL$ as follows:
\begin{equation}\label{eq:m2lind2}
 \cL(X) :=\frac12\{V^*V,X\}- V^*XV.
\end{equation}
The semigroup generated by $\cL$ is denoted by $\Phi_t=e^{-t\cL}$.

\begin{proposition}\label{prop:qubit-kms-primitive}
The semigroup $\{\Phi_t\}_{t\geq 0}$ with generator \eqref{eq:m2lind2} is primitive KMS-symmetric with respect to $\rho$. 
Moreover, the semigroup is not GNS-symmetric.
\end{proposition}

\begin{proof}
Direct multiplication gives
\begin{equation}\label{eq:vh}
 VD_\rho^{1/2}=D_\rho^{1/2}V^*,
 \qquad
 V^*V=\begin{pmatrix}4&0\\0&1\end{pmatrix},
 \qquad [V^*V,D_\rho^{1/2}]=0.
\end{equation}
The Hilbert--Schmidt adjoint of $\cL$ is
\begin{align*}
 \cL^*(X)=\frac12\{V^*V,X\}-VXV^*.
\end{align*}
Using \eqref{eq:vh},
\begin{align*}
 \cL^*(D_\rho^{1/2}XD_\rho^{1/2})=D_\rho^{1/2} \left(\frac12\{V^*V,X\}-V^*XV\right)D_\rho^{1/2}=D_\rho^{1/2}\cL(X)D_\rho^{1/2}.
\end{align*}
This shows that the semigroup is KMS symmetric.  It also gives
$\cL^*(D_\rho)=0$. 

Suppose that 
\begin{align*}
P=\begin{pmatrix}0&1\\1&0\end{pmatrix},
 \qquad
Q=\begin{pmatrix}0&-i\\i&0\end{pmatrix},
 \qquad W=-iPQ, \qquad
WV^*V=\begin{pmatrix}4&0\\0&-1\end{pmatrix}.
\end{align*}
A direct calculation gives
\begin{equation}\label{eq:qubitspectrum}
\cL(I)=0,
 \qquad
\cL(P)=\frac{1}{2}P,
 \qquad
\cL (Q)=\frac{9}{2}Q,
 \qquad
\cL(WV^*V)=5WV^*V.
\end{equation}
The eigenvalues of $\cL$ are $\displaystyle 0,\frac12, \frac92,5$.
This shows that $\ker\cL=\bC$.
Hence the quantum Markov semigroup is primitive.

Finally, for the GNS inner product
$\langle X,Y\rangle_{\GNS,\rho}=\Tr(D_\rho X^*Y)$, one finds
\begin{align*}
 \langle E_{12}, \cL(E_{21})\rangle_{\GNS,\rho}=-\frac85,
 \qquad
 \langle \cL (E_{12}),E_{21}\rangle_{\GNS,\rho}=-\frac25.
\end{align*}
Thus the semigroup is not GNS-symmetric.
\end{proof}

The next theorem explains why a non-self-adjoint direction can obstruct complete MLSI even though ordinary MLSI only sees self-adjoint density perturbations.

\begin{theorem}\label{lem:cmlsibd}
\label{lem:endpoint-obstruction}
Let $\Phi_t=e^{-t\cL}$ be a primitive quantum Markov semigroup on
$M_d(\bC)$, where $d\geq2$, with faithful invariant state $\rho$.
Suppose that $\cL$ has the GKSL representation \eqref{eq:gksl}.

Define
\[
\mathcal C_V
=\{X\in M_d(\mathbb C):
[V_j,X]=0 \text{ for all }1\leq j\leq m\}
\]
and
\[
\nu_\rho(\cL)
=
\inf_{\substack{X\neq0\\\operatorname{Tr}(D_\rho X)=0}}
\frac{
\operatorname{Re}\operatorname{Tr}(D_\rho X^*\cL(X))
}{
\operatorname{Tr}(D_\rho X^*X)
}.
\]
Then
\[
\nu_\rho(\cL)=0
\quad\Longleftrightarrow\quad
\mathcal C_V\neq\mathbb CI.
\]
Moreover, we have 
$0\leq\alpha_{\mathrm{CMLSI}}(\cL)
\leq\nu_\rho(\cL)$.

Consequently, if there exists a non-scalar matrix $X_0$
such that
$[V_j,X_0]=0$
 for $1\leq j\leq m$,
then
$\alpha_{\mathrm{CMLSI}}(\cL)=0$.
In particular, the semigroup does not satisfy CMLSI with
any positive constant, and an $M_2(\mathbb C)$ ancilla
already witnesses this failure. No KMS-symmetry assumption
is needed, and $X_0$ need not be self-adjoint.

\end{theorem}

\begin{proof}
A direct calculation gives
\[
\cL(X^*)X+X^*\cL(X)-\cL(X^*X)
=
\sum_{j=1}^m[V_j,X]^*[V_j,X].
\]
Using $\cL^*(D_\rho)=0$, we obtain
\begin{equation}
\label{eq:noise-commutant-gns-energy}
2\operatorname{Re}\operatorname{Tr}(D_\rho X^*\cL(X))
=
\sum_{j=1}^m
\operatorname{Tr}
\bigl(D_\rho[V_j,X]^*[V_j,X]\bigr).
\end{equation}
Thus $\nu_\rho(\cL)\geq0$.

If $X_0\in\mathcal C_V\setminus\mathbb CI$, then
\[
X=X_0-\operatorname{Tr}(D_\rho X_0)I
\]
is nonzero, satisfies $\operatorname{Tr}(D_\rho X)=0$,
and commutes with every $V_j$. Hence
\eqref{eq:noise-commutant-gns-energy} gives
$\nu_\rho(\cL)=0$.

Conversely, the infimum defining $\nu_\rho(\cL)$ is
attained on the compact set
\[
\left\{
X\in M_d(\bC):\operatorname{Tr}(D_\rho X)=0,\qquad
\operatorname{Tr}(D_\rho X^*X)=1
\right\}.
\]
If the infimum is zero, faithfulness of $D_\rho$ and
\eqref{eq:noise-commutant-gns-energy} imply that a minimizer
commutes with every $V_j$. This minimizer is non-scalar,
which proves the asserted equivalence.

Let $\alpha_2(\cL)$ denote the optimal MLSI constant of
the amplified semigroup $\{\Phi_t\otimes\Id_2\}_{t\geq0}$.
Since $\alpha_{\mathrm{CMLSI}}(\cL)\leq\alpha_2(\cL)$,
it suffices to prove $\alpha_2(\cL)\leq\nu_\rho(\cL)$.
Suppose that $0<r<1$ and 
\begin{align*}
T_r=\frac1{1+r}\begin{pmatrix}r&0\\0&1\end{pmatrix},
 \qquad
\widetilde{T}_r=D_\rho \otimes T_r.
\end{align*}
Let
\begin{align*}
 Y=X\otimes E_{01}+X^*\otimes E_{10},
 \qquad
 S_r=\K_{\widetilde{T}_r}(Y).
\end{align*}
Here $\{E_{ab}\}_{a,b\in\{0,1\}}$ are the standard matrix units of $M_2(\bC)$.
By the fact that  $\Tr(D_\rho X)=0$, we see that $\Tr_A (S_r)=0.$
Thus, for sufficiently small real $\varepsilon$, the element $\widetilde{T}_{\varepsilon,r}=\widetilde{T}_r+\varepsilon S_r$ is a faithful density matrix with ancillary marginal $T_r$.

The perturbation expansions
		\eqref{eq:entropyest}--\eqref{eq:informationest} give the estimates of the relative entropy and Fisher information as follows
		\begin{align*}
			2H(\widetilde T_{\varepsilon,r}\|\widetilde T_r)
			&=\varepsilon^2\int_0^1
			\Tr\!\left(\widetilde T_r^sY
			\widetilde T_r^{1-s}Y\right)\,ds
			+\bO(\varepsilon^3)\\
			&=\frac{\varepsilon^2}{1+r}\int_0^1 r^s
			\Tr\!\left(D_\rho^sX D_\rho^{1-s}X^*\right)\,ds\\
			&\quad+\frac{\varepsilon^2}{1+r}\int_0^1 r^{1-s}
			\Tr\!\left(D_\rho^sX^*D_\rho^{1-s}X\right)\,ds
			+\bO(\varepsilon^3)\\
			&=\frac{2\varepsilon^2}{1+r}
			\int_0^1 r^{1-s}
			\Tr\!\left(D_\rho^sX^*D_\rho^{1-s}X\right)\,ds
			+\bO(\varepsilon^3),
\end{align*}
and		
\begin{align*}		
			\cI_{\cL\otimes\Id_2}(\widetilde T_{\varepsilon,r})
			&=\varepsilon^2\int_0^1\operatorname{Re}\Tr\!\left(
			\widetilde T_r^sY\widetilde T_r^{1-s}
			(\cL\otimes\Id_2)(Y)\right)\,ds
			+\bO(\varepsilon^3)\\
			&=\frac{\varepsilon^2}{1+r}\int_0^1 r^s
			\operatorname{Re}\Tr\!\left(
			D_\rho^sX D_\rho^{1-s}\cL(X)^*\right)\,ds\\
			&\quad+\frac{\varepsilon^2}{1+r}\int_0^1 r^{1-s}
			\operatorname{Re}\Tr\!\left(
			D_\rho^sX^*D_\rho^{1-s}\cL(X)\right)\,ds
			+\bO(\varepsilon^3)\\
			&=\frac{2\varepsilon^2}{1+r}
			\int_0^1 r^{1-s}\operatorname{Re}\Tr\!\left(
			D_\rho^sX^*D_\rho^{1-s}\cL(X)\right)\,ds
			+\bO(\varepsilon^3).
		\end{align*}
		The quadratic coefficient in the expansion of
$2H(\widetilde T_{\varepsilon,r}\|\widetilde T_r)$ is strictly positive for $X\neq0$.
Dividing the two expansions and letting $\varepsilon\to0$ therefore yields
		
\begin{equation}\label{eq:ratio-general}
 \lim_{\varepsilon\to0}
 \frac{\cI_{\cL\otimes\Id_2}(\widetilde{T}_{\varepsilon,r})}
      {2H(\widetilde{T}_{\varepsilon,r}\| \widetilde{T}_r)}
 =h_r(X),
\end{equation}
where
\begin{equation}\label{eq:Rrx}
h_r(X)=
 \frac{\displaystyle\int_0^1 r^{1-s} \operatorname{Re}\Tr\!\left(
D_\rho^s X^*D_\rho^{1-s}\cL(X)\right)\,ds}
      {\displaystyle\int_0^1 r^{1-s}\Tr\!\left(D_\rho^s X^*D_\rho^{1-s}X\right)\,ds}.
\end{equation}
Indeed, the two off-diagonal ancilla blocks give equal contributions after
the change of variables $s\mapsto1-s$.

If $\{\Phi_t\otimes\Id_2\}_{t\geq0}$ satisfies MLSI$(\beta)$,
then $\beta\leq h_r(X)$ for every $r\in(0,1)$.  
The probability measures proportional to $r^{1-s}ds$ converge weakly to the point mass at $s=1$ as $r\downarrow0$.  
Therefore
\begin{align*}
 \beta \leq\lim_{r\downarrow0}h_r(X)
 =\frac{\operatorname{Re}\Tr(D_\rho X^*\cL(X))}
        {\Tr(D_\rho X^*X)},
\end{align*}
which proves the claim.

\end{proof}

\begin{corollary}[The single-noise case]
\label{cor:single-noise-no-cmlsi}
Let $\Phi_t=e^{-t\cL}$ be a primitive quantum Markov
semigroup on $M_d(\mathbb C)$, $d\geq2$, with faithful
invariant density matrix $D_\rho$. If
\[
\cL(X)
=i[H,X]+\frac12\{V^*V,X\}-V^*XV
\]
for a non-scalar matrix $V$, then
\[
\alpha_{\mathrm{CMLSI}}(\cL)=0.
\]
\end{corollary}

\begin{proof}
Take $X_0=V$. Then $X_0$ is non-scalar and
$[V,X_0]=0$, so
Theorem~\ref{lem:cmlsibd} applies.
\end{proof}

\begin{remark}
The condition $[V_j,X_0]=0$ involves only the noise
operators; it imposes no commutation requirement with
$V_j^*$ or $H$.
A zero GNS numerical dissipation direction need not
belong to $\ker\cL$, so the condition is compatible
with primitivity.
The theorem gives a sufficient condition for failure
of CMLSI; it does not assert its necessity or failure
of ordinary MLSI.
\end{remark}

\begin{remark}\label{thm:m2}
Corollary \ref{cor:single-noise-no-cmlsi} yields
 that the primitive KMS-symmetric quantum Markov semigroup $\{\Phi_t\}_{t\geq 0}$ with generator \eqref{eq:m2lind2} does not have a complete modified log Sobolev inequality.
Precisely, $ \alpha_{\mathrm{CMLSI}}(\cL)=0$.
An $M_2(\bC)$ ancilla already witnesses the failure.
\end{remark}



\begin{remark}
Gao and Rouz\'e prove that finite-dimensional GNS-symmetric semigroups satisfy
a complete modified logarithmic Sobolev inequality
\cite{GaoRou22}.  
The preceding example is KMS-symmetric but not
GNS-symmetric, so it exhibits a genuine separation between these two symmetry notions at the complete-MLSI level.
\end{remark}

\begin{corollary}
The non-primitive KMS-symmetric quantum Markov semigroup $\{\Phi_t\otimes \Id_{M_2(\bC)}\}_{t\geq 0}$ given by \eqref{eq:m2lind2} does not have a modified log Sobolev inequality.
\end{corollary}
\begin{proof}
It follows from the proof of Theorem~\ref{lem:cmlsibd}, which shows that
an $M_2(\bC)$ ancilla already witnesses the failure.
\end{proof}

\section{CMLSI for Graph-based KMS Symmetric QMS}

In this section, we shall investigate a primitive KMS-symmetric quantum Markov semigroup arising from an undirected graph and obtain the complete MLSI.

Suppose that $\cG=(V(\cG), E(\cG))$ with the vertex set $V(\cG)=\{1,2, \ldots, n\}$ is a finite simple undirected graph with a positive weight $\gamma_{jk}>0$ on the edge $e_{jk}$ in $E(\cG)$, where $\gamma_{jk}=\gamma_{kj}$. 
We set $\gamma_{jk}=0$ whenever $e_{jk}\notin E(\cG)$.

Let $\{E_{jk}\}_{j,k=1}^n$ be the system of matrix units of $M_n(\bC)$.
Let $\rho$ be a faithful state on $M_n(\bC)$ with density $\displaystyle D_\rho=\sum_{j=1}^n s_j E_{jj}$.
Set
\[
\gamma_{\min}=\min_{e_{jk}\in E(\cG)}\gamma_{jk},
\qquad
s_{\min}=\min_{1\leq j\leq n}s_j.
\]

Let $d_j$ be the weighted degree with respect to $\rho$ at vertex $j$, which is defined as follows:
\begin{align*}
    d_j=\sum_{k: e_{jk}\in E(\cG)}\gamma_{jk} s_k, \quad d_{\max}=\max_j\{d_j\}.
\end{align*}
Let $\vec{s}=(s_1, \ldots, s_n)^{\mathsf{T}}$ and
\begin{align*}
\varphi_{\cG}=\inf_{\vec{s}\perp \vec{t},\  \vec{t}\neq 0}\frac{\displaystyle \sum_{e_{jk}\in E(\cG)}\gamma_{jk}s_js_k |t_j-t_k|^2}{\displaystyle \sum_{j=1}^n s_j|t_j|^2},
\end{align*}
where $\vec{t}=(t_1, \ldots, t_n)^{\mathsf{T}}$.
We have that 
\begin{align*}
 \frac{\gamma_{\min}\min_{e_{jk}}s_js_k}{\max_j s_j}   \lambda_{\cG}\leq \varphi_{\cG}\leq 2d_{\max},
\end{align*}
where $\lambda_\cG$ is the spectral gap of the unnormalized Laplacian of the unweighted graph $\cG$.

For each edge $e_{jk}\in E(\cG)$, we set
\begin{align}
V_{jk}=\sqrt{s_j}E_{jk}+\sqrt{s_k} E_{kj}\in M_n(\bC).
\end{align}
The Lindbladian associated to the graph $\cG$ is given by
\begin{align}\label{eq:mnlind1}
    \cL(X) =\sum_{e_{jk}\in E(\cG)} \gamma_{jk} \left(\frac{1}{2}\{V_{jk}^*V_{jk}, X\}- V_{jk}^* X V_{jk}\right), \quad X\in M_n(\bC).
\end{align}
By the GKSL representation of the generator of the quantum Markov semigroup \cite{Lin76}, we see that the corresponding semigroup is a quantum Markov semigroup, denoted by $\{\Phi_t\}_{t\geq 0}$.
Moreover, we have that 
\begin{align}
\cL(E_{jj}) =& \sum_{e_{jk}\in E(\cG)} \gamma_{jk} (s_k E_{jj}-s_jE_{kk}), \label{eq:lindbase1}\\
    \cL(E_{jk})= & \frac{d_j+d_k}{2} E_{jk}-\mathbbm{1}_{e_{jk}\in E(\cG)} \gamma_{jk}\sqrt{s_js_k}E_{kj}, \quad j\neq k. \label{eq:lindbase2}
\end{align}

\begin{proposition}
The semigroup $\{\Phi_t\}_{t\geq 0}$ with the Lindbladian \eqref{eq:mnlind1} is a KMS-symmetric quantum Markov semigroup with respect to $\rho$.
If there exists an edge $e_{jk}\in E(\cG)$ such that $s_j \neq s_k$, then $\{\Phi_t\}_{t\geq 0}$ is not GNS-symmetric.
\end{proposition}
\begin{proof}
By a direct computation, we have that 
\begin{align*}
D_\rho^{1/2}V_{jk}^*=V_{jk}D_\rho^{1/2}, \quad V_{jk}^* V_{jk}=s_k E_{jj} +s_jE_{kk}.
\end{align*}
This shows that $[V_{jk}^*V_{jk}, D_\rho]=0$ and $V_{jk}D_\rho V_{jk}^*=V_{jk}^*V_{jk}D_\rho$. 
Hence for any $X\in M_n(\bC)$, 
\begin{align*}
    \cL^*(D_\rho^{1/2} X D_\rho^{1/2})=D_\rho^{1/2}\cL(X) D_\rho^{1/2},
\end{align*}
which implies that $\{\Phi_t\}_{t\geq 0}$ is KMS-symmetric with respect to $\rho$.
Note that 
\begin{align*}
\Tr(\cL(E_{jk}) D_\rho E_{jk})-\Tr(D_\rho \cL(E_{jk})E_{jk})=\gamma_{jk}\sqrt{s_js_k} \left(s_k-s_j \right).
\end{align*}
If there exists an edge $e_{jk}$ with $s_j\neq s_k$, we see that $\{\Phi_t\}_{t\geq 0}$ is not GNS-symmetric.

\end{proof}

\begin{lemma}
We have that 
\begin{align*}
\|\cL\|_{\BKM, \rho}  \leq 2 d_{\max}.
\end{align*}
\end{lemma}
\begin{proof}
Let $\cH_0=\text{span}\{E_{jj}:j=1, \ldots, n\}$ and $\cH_{jk}=\text{span}\{E_{jk}, E_{kj}\}$ with $j<k$.
Then by Equation \eqref{eq:kd}, we see that
\begin{align*}
M_n(\bC)=\cH_0\bigoplus\bigoplus_{j<k} \cH_{jk},
\end{align*}
where the orthogonality is taken with respect to $\langle \cdot,\cdot \rangle_{ \rho}$.
By Equations \eqref{eq:lindbase1} and \eqref{eq:lindbase2},
the subspaces $\cH_0$ and $\cH_{jk}$ are invariant under $\cL$.

The maximal eigenvalue of $\cL$ on $\cH_{jk}$ is 
\begin{align*}
\frac{d_j+d_k}{2}+\mathbbm{1}_{e_{jk}\in E(\cG)} \gamma_{jk}\sqrt{s_js_k}\leq 2d_{\max}.
\end{align*}
The maximal eigenvalue of $\cL$ on $\cH_0$ is \begin{align*}
    \sup_{\vec{s}\perp \vec{t},\  \vec{t}\neq 0}\frac{\displaystyle \sum_{e_{jk}\in E(\cG)}\gamma_{jk}s_js_k |t_j-t_k|^2}{\displaystyle \sum_{j=1}^n s_j|t_j|^2} \leq 2d_{\max}.
\end{align*}
This shows that
\[
\|\mathcal L\|_{\mathrm{BKM},\rho}
=
\max\left\{
\lambda_{\max}\!\left(
\left.\mathcal L\right|_{\mathcal H_0}
\right),
\max_{j<k}\left\{
\frac{d_j+d_k}{2}
+
1_{e_{jk}\in E(G)}\gamma_{jk}\sqrt{s_js_k}
\right\}
\right\}
\le 2d_{\max}.
\]
This completes the proof of the lemma.

\end{proof}

\begin{lemma}\label{lem:lowerbd1}
Suppose that $G$ is connected and $n\geq 3$. Then we have that
\begin{align*}
 \inf_{X\in \cS_\rho} Q(X)=& \min\left\{\varphi_{\cG}, \min_{j<k}\left\{ \frac{d_j+d_k}{2}-\mathbbm{1}_{e_{jk}\in E(\cG)} \gamma_{jk}\sqrt{s_js_k}\right\}\right\} \\
 \geq &   \frac{1}{2}\gamma_{\min}s_{\min}^2 \lambda_{\cG}.
\end{align*}
\end{lemma}
\begin{proof}
Note that the minimal eigenvalue of $\cL$ on $\cH_{jk}$ is 
\begin{align*}
\frac{d_j+d_k}{2}-\mathbbm{1}_{e_{jk}\in E(\cG)} \gamma_{jk}\sqrt{s_js_k}.
\end{align*}
The smallest positive eigenvalue of $\cL$ on $\cH_0$ is $\varphi_{\cG}$.
Note that 
\begin{align*}
    \frac{d_j+d_k}{2}-\mathbbm{1}_{e_{jk}\in E(\cG)} \gamma_{jk}\sqrt{s_js_k}
    \geq  \frac{d_j+d_k}{2}-\frac{1}{2}\mathbbm{1}_{e_{jk}\in E(\cG)} \gamma_{jk}(s_j+s_k)
    \geq \frac{1}{2}\gamma_{\min}s_{\min}
\end{align*}
and 
\begin{align*}
    \varphi_{\cG} \geq \gamma_{\min}s_{\min}^2 \lambda_{\cG}.
\end{align*}
Since $\cG$ is a simple graph on $n$ vertices, its unnormalized
Laplacian satisfies $\lambda_{\cG}\leq n$.
Also, $\sum_{j=1}^n s_j=1$ gives $s_{\min}\leq 1/n$.
Thus $s_{\min}\lambda_{\cG}\leq 1$, and we obtain
\begin{align*}
\inf_{X\in \cS_\rho} Q(X) \geq \frac{1}{2}\gamma_{\min}s_{\min}^2 \lambda_{\cG}.
\end{align*}
This completes the proof of the lemma.
\end{proof}

\begin{proposition}
Suppose that $\cG$ is connected and $n\geq2$.
If $n\geq3$, the semigroup $\{\Phi_t\}_{t\geq0}$ given by
\eqref{eq:mnlind1} is primitive.
If $n=2$, it is primitive if and only if $s_1\neq s_2$.
\end{proposition}
\begin{proof}
Suppose that $X_0\in M_n(\bC)$ with $\cL(X_0)=0$.
For any $X\in M_n(\bC)$, the gradient form is
\begin{align*}
    \Gamma(X)=X^*\cL(X)+\cL(X)^*X-\cL(X^*X)=\sum_{e_{jk}\in E(\cG)} \gamma_{jk} [V_{jk}, X]^*[V_{jk}, X].
\end{align*}
Since $\cL(X_0)=0$ and $\rho$ is invariant,
\[
\Tr\!\left(D_\rho\Gamma(X_0)\right)=0.
\]
As $\Gamma(X_0)\geq0$ and $D_\rho>0$, it follows that $\Gamma(X_0)=0$.
Hence $[V_{jk},X_0]=0$ for every $e_{jk}\in E(\cG)$.
If $n=2$ and $\cG$ is connected, a direct computation using
\eqref{eq:lindbase1} and \eqref{eq:lindbase2} gives the eigenvalues
\[
0,\qquad\gamma_{12},\qquad
\gamma_{12}\left(\frac12\pm\sqrt{s_1s_2}\right).
\]
Since $s_1+s_2=1$, it follows that
\[
\ker\cL=\bC I
\quad\Longleftrightarrow\quad s_1\neq s_2.
\]

Suppose that $\cG$ is connected and $n \geq 3$.
We have that the eigenvalues
\begin{align*}
\frac{d_j+d_k}{2}-\mathbbm{1}_{e_{jk}\in E(\cG)} \gamma_{jk}\sqrt{s_js_k} 
=\frac{\gamma_{jk}}{2}(\sqrt{s_j}-\sqrt{s_k})^2 +\frac{1}{2}\left(\sum_{e_{\ell j}, \ell\neq k}\gamma_{j\ell}s_{\ell} + \sum_{e_{\ell k}, \ell\neq j}\gamma_{k\ell}s_{\ell}\right)\neq 0.
\end{align*}
This shows that $\ker\cL=\bC$ and the semigroup is primitive.
\end{proof}

\begin{lemma}\label{lem:graph-complete-bkm}
Suppose that $\cG$ is connected and $n\geq3$.
Let $m\geq1$, let $\sigma$ be a faithful state on $M_m(\bC)$,
and set $\eta=\rho\otimes\sigma$.
Then
\[
\|\cL\otimes\Id_m\|_{\BKM,\eta}\leq2d_{\max}.
\]
Define
\[
\cS_{\rho\otimes\Id_m}
=\left\{
X=X^*\in M_n(\bC)\otimes M_m(\bC):
(\rho\otimes\Id_m)(X)=0,\ \|X\|=1
\right\}.
\]
Here $Q$ is defined as in \eqref{eq:keybd}, with $\cL$ and $D_\rho$
replaced by $\cL\otimes\Id_m$ and $D_\eta=D_\rho\otimes D_\sigma$,
respectively. We have
\begin{align*}
\inf_{X\in\cS_{\rho\otimes\Id_m}}Q(X)
&\geq\min\left\{\varphi_{\cG},
\min_{j<k}\left\{
\frac{d_j+d_k}{2}
-\frac12\mathbbm{1}_{e_{jk}\in E(\cG)}\gamma_{jk}(s_j+s_k)
\right\}\right\}\\
&\geq\frac12\gamma_{\min}s_{\min}^2\lambda_{\cG}.
\end{align*}
\end{lemma}

\begin{proof}
Choose matrix units $\{F_{ab}\}_{a,b=1}^m$ such that
$D_\sigma=\sum_{a=1}^m t_aF_{aa}$, where $t_a>0$.
The matrix units $E_{jk}\otimes F_{ab}$ are orthogonal for the
BKM inner product induced by $\eta$, with squared norms
\[
p_{j,k,a,b}
=\int_0^1(s_jt_a)^u(s_kt_b)^{1-u}\,du
=\begin{cases}
\displaystyle\frac{s_jt_a-s_kt_b}{\log(s_jt_a)-\log(s_kt_b)},
&s_jt_a\neq s_kt_b,\\[4pt]
s_jt_a,&s_jt_a=s_kt_b.
\end{cases}
\]
By \eqref{eq:lindbase1} and \eqref{eq:lindbase2}, the mutually
orthogonal subspaces
\[
\cH_{0,ab}=\spanop\{E_{jj}\otimes F_{ab}:1\leq j\leq n\},
\qquad
\cH_{jk,ab}=\spanop\{E_{jk}\otimes F_{ab},E_{kj}\otimes F_{ab}\}
\]
are invariant under $\cL\otimes\Id_m$, where $j<k$ and
$1\leq a,b\leq m$.

First consider $\cH_{0,ab}$. Set
$\ell_{ab}=\int_0^1t_a^ut_b^{1-u}\,du>0$.
Since $p_{j,j,a,b}=s_j\ell_{ab}$, the BKM inner product on this
block is a positive scalar multiple of the weighted inner product
on $\cH_0$. Thus the restriction of $\cL\otimes\Id_m$ is
self-adjoint on this block and has norm at most $2d_{\max}$,
by the same estimate as for $\cL$ on $\cH_0$.
For $Z=\sum_j z_jE_{jj}\otimes F_{ab}$ with $\sum_j s_jz_j=0$,
we also have
\[
\operatorname{Re}\langle Z,(\cL\otimes\Id_m)(Z)\rangle_{\BKM,\eta}
=\ell_{ab}\sum_{e_{jk}\in E(\cG)}
\gamma_{jk}s_js_k|z_j-z_k|^2
\geq\varphi_{\cG}\|Z\|_{\BKM,\eta}^2.
\]

Next fix an off-diagonal block $\cH_{jk,ab}$ and write
\[
a_0=\frac{d_j+d_k}{2},\qquad
b_0=\mathbbm{1}_{e_{jk}\in E(\cG)}\gamma_{jk}\sqrt{s_js_k},
\qquad
u_0=\sqrt{\frac{p_{j,k,a,b}}{p_{k,j,a,b}}}.
\]
In the BKM-orthonormal basis obtained by normalizing
$E_{jk}\otimes F_{ab}$ and $E_{kj}\otimes F_{ab}$,
the matrix of the restriction is
\[
A=\begin{pmatrix}a_0&-b_0u_0\\-b_0u_0^{-1}&a_0\end{pmatrix}.
\]
The smallest eigenvalue of its Hermitian part $(A+A^*)/2$ is
\begin{equation}\label{eq:bd34}
\frac{d_j+d_k}{2}
-\frac12\mathbbm{1}_{e_{jk}\in E(\cG)}\gamma_{jk}\sqrt{s_js_k}
\left(
\sqrt{\frac{p_{j,k,a,b}}{p_{k,j,a,b}}}
+\sqrt{\frac{p_{k,j,a,b}}{p_{j,k,a,b}}}
\right),
\end{equation}
and the largest eigenvalue of the same Hermitian part is
\begin{equation}\label{eq:bd35}
\frac{d_j+d_k}{2}
+\frac12\mathbbm{1}_{e_{jk}\in E(\cG)}\gamma_{jk}\sqrt{s_js_k}
\left(
\sqrt{\frac{p_{j,k,a,b}}{p_{k,j,a,b}}}
+\sqrt{\frac{p_{k,j,a,b}}{p_{j,k,a,b}}}
\right).
\end{equation}
The integral formula for $p_{j,k,a,b}$ gives
\[
\min\left\{\frac{s_j}{s_k},\frac{s_k}{s_j}\right\}
\leq\frac{p_{j,k,a,b}}{p_{k,j,a,b}}
\leq\max\left\{\frac{s_j}{s_k},\frac{s_k}{s_j}\right\}.
\]
Consequently,
$u_0+u_0^{-1}\leq\sqrt{s_j/s_k}+\sqrt{s_k/s_j}$.
The smallest eigenvalue in \eqref{eq:bd34} is therefore at least
\[
\frac{d_j+d_k}{2}
-\frac12\mathbbm{1}_{e_{jk}\in E(\cG)}\gamma_{jk}(s_j+s_k).
\]
For the operator norm, we use the matrix $A$ itself:
\[
\begin{aligned}
\|A\|_{2\to2}
&\leq a_0+b_0\max\{u_0,u_0^{-1}\}\\
&\leq\frac{d_j+d_k}{2}
+\mathbbm{1}_{e_{jk}\in E(\cG)}\gamma_{jk}\max\{s_j,s_k\}\\
&\leq2d_{\max}.
\end{aligned}
\]
Combining this with the diagonal-block estimate proves
$\|\cL\otimes\Id_m\|_{\BKM,\eta}\leq2d_{\max}$.

Finally, if $(\rho\otimes\Id_m)(X)=0$, the coefficients of each
diagonal block of $X$ satisfy $\sum_j s_jz_{j,ab}=0$.
Thus the diagonal-block lower bound applies to every such component;
the off-diagonal components satisfy the bound from \eqref{eq:bd34}.
By orthogonality, these estimates give the first lower bound for $Q$.
Since $\cG$ is connected and $n\geq3$, for every $j<k$ we have
\[
\begin{aligned}
\frac{d_j+d_k}{2}
-\frac12\mathbbm{1}_{e_{jk}\in E(\cG)}\gamma_{jk}(s_j+s_k)
&=\frac12\left(
\sum_{\ell\neq j,k}\gamma_{j\ell}s_\ell
+\sum_{\ell\neq j,k}\gamma_{k\ell}s_\ell
\right)\\
&\geq\frac12\gamma_{\min}s_{\min}.
\end{aligned}
\]
Together with $\varphi_{\cG}\geq\gamma_{\min}s_{\min}^2\lambda_{\cG}$
and $s_{\min}\lambda_{\cG}\leq1$, this proves the second lower bound.
\end{proof}

\begin{theorem}\label{thm:mlsigraph}
Suppose that $\cG$ is connected and $n \geq 3$.
We have that 
\begin{align*}
  \frac{\gamma_{\min}^2 s_{\min}^5\lambda_{\cG}^2}{112d_{\max}}  \leq \alpha_{MLSI} \leq \min\left\{\varphi_{\cG}, \min_{j<k}\left\{ \frac{d_j+d_k}{2}-\mathbbm{1}_{e_{jk}\in E(\cG)} \gamma_{jk}\sqrt{s_js_k}\right\}\right\},
\end{align*}
and 
\begin{align*}
 \frac{\gamma_{\min}^2 s_{\min}^5\lambda_{\cG}^2}{112n d_{\max}} \leq \alpha_{CMLSI}\leq \min\left\{\varphi_{\cG}, \min_{j<k}\left\{ \frac{d_j+d_k}{2}-\mathbbm{1}_{e_{jk}\in E(\cG)} \gamma_{jk}\sqrt{s_js_k}\right\}\right\}.
\end{align*}
If $n=2$, $\cG$ is connected, and $s_1<s_2$, with edge weight
$\gamma_{12}=\gamma$, we have $\alpha_{CMLSI}=0$ and
\begin{align*}
\alpha_{MLSI}\geq \frac{\gamma s_1}{14s_2}\left(\frac{1}{2}-\sqrt{s_1s_2}\right)^2.
\end{align*}
\end{theorem}
\begin{proof}
By Equation \eqref{eq:mlsiest0} in the proof of Theorem \ref{thm:mlsiprim}, we have that 
\begin{align*}
    \alpha_{MLSI} \geq&  \frac{\left(\inf_{X\in \cS_\rho}Q(X)\right)^2}{14(\|D_\rho^{-1}\|-1)\|\cL\|_{\BKM, \rho }} \\
 \geq &\frac{\gamma_{\min}^2 s_{\min}^5\lambda_{\cG}^2}{112 d_{\max}},
\end{align*}
where $\kappa=\|D_\rho^{-1}\|=s_{\min}^{-1}$ is the inverse of the minimal eigenvalue of $D_\rho$.

Now we consider the completely modified logarithmic Sobolev inequality.
For a faithful density matrix $D\in M_n(\bC)\otimes M_m(\bC)$,
let $D_R=\operatorname{Tr}_{M_n}(D)$ and let $\sigma$ be the
faithful state with density $D_R$. Set $\eta=\rho\otimes\sigma$,
so that $D_\eta=D_\rho\otimes D_R$. Then
\[
D\leq\left(\sum_{j=1}^n s_j^{-1}\right)D_\eta
\leq n s_{\min}^{-1}D_\eta.
\]
For every $X\in M_n(\bC)\otimes M_m(\bC)$, we have
\[
\operatorname{Tr}_{M_n}\!\left(\K_{D_\eta}(X)\right)
=\K_{D_R}\!\left((\rho\otimes\Id_m)(X)\right).
\]
Since $\K_{D_R}$ is invertible, this identity shows that
$X=\K_{D_\eta}^{-1}(D-D_\eta)$ satisfies
$(\rho\otimes\Id_m)(X)=0$.
We may therefore apply the argument of Theorem~\ref{thm:mlsiprim}
on the affine space of density matrices with fixed marginal $D_R$,
using Lemma~\ref{lem:graph-complete-bkm} and
$\kappa=n s_{\min}^{-1}$.
This proves the stated lower bound for $\alpha_{CMLSI}$,
uniformly in $m$ and $D_R$.
For a general density matrix, restrict the ancillary space to
$\operatorname{supp}D_R$ and approximate $D$ by
$(1-\varepsilon)D+\varepsilon D_\eta$;
the entropy contraction inequality passes to the limit.

By Lemma~\ref{lem:tech2}, for every $X\in\cS_\rho$ and
$D_\varepsilon=D_\rho+\varepsilon\K_{D_\rho}(X)$, we have
\[
\alpha_{MLSI}
\leq\lim_{\varepsilon\downarrow0}
\frac{\cI_{\cL}(D_\varepsilon)}
     {2H(D_\varepsilon\|D_\rho)}
=Q(X).
\]
Taking the infimum and applying Lemma~\ref{lem:lowerbd1}
gives the stated upper bound for $\alpha_{MLSI}$.
The same upper bound for $\alpha_{CMLSI}$ follows from
$\alpha_{CMLSI}\leq\alpha_{MLSI}$.

If $n=2$, we have that 
\begin{align*}
 \kappa-1=s_1^{-1}-1=\frac{s_2}{s_1}, \quad \inf_{X\in \cS_\rho} Q(X) =\gamma \left(\frac{1}{2}-\sqrt{s_1s_2}\right), \quad \|\cL\|_{\BKM, \rho}=\gamma.
 \end{align*}
Substituting these values into the bound from
Theorem~\ref{thm:mlsiprim} gives
\[
\alpha_{MLSI}
\geq\frac{\left(\inf_{X\in\cS_\rho}Q(X)\right)^2}
{14\|\cL\|_{\BKM,\rho}(\kappa-1)}
=\frac{\gamma s_1}{14s_2}
\left(\frac12-\sqrt{s_1s_2}\right)^2.
\]
For the complete setting, Corollary~\ref{cor:single-noise-no-cmlsi}
applies to the single noise operator $\sqrt{\gamma}V_{12}$ and gives
$\alpha_{CMLSI}=0$.
\end{proof}

\section{CMLSI for Bimodule KMS-Symmetric QMS}

In this section, we show that the primitive bimodule KMS-symmetric
quantum Markov semigroups considered in Section 7.2 of \cite{JWW26}
satisfy CMLSI, even when they are not KMS-symmetric with respect
to their invariant states.

Suppose that $P_1,\ldots,P_n,Q_1,\ldots,Q_n$ are unitary generators subject to
\begin{equation}\label{eq:CAR-PQ}
 \{P_j,P_k\}=\{Q_j,Q_k\}=2\delta_{jk}I, \qquad \{Q_j,P_k\}=0,\qquad 1\leq j,k\leq n.
\end{equation}
Then the C$^*$ algebra $\mathcal{M}=C^*(P_1,\ldots,P_n,Q_1,\ldots,Q_n)\cong M_{2^n}(\mathbb{C})=M_d(\bC)$.
Let $\tau=d^{-1}\Tr$ denote the normalized trace on $\mathcal M$.
For example, using the usual Pauli matrices, one may take
\[
 Q_j=\sigma_z^{\otimes(j-1)}\otimes\sigma_x\otimes I_2^{\otimes(n-j)},
 \qquad
 P_j=\sigma_z^{\otimes(j-1)}\otimes\sigma_y\otimes I_2^{\otimes(n-j)},
\]
where $\sigma_x, \sigma_y, \sigma_z$ are Pauli matrices in $M_2(\bC)$.

As in \cite[Section~7.1]{JWW26}, define
\begin{equation}\label{eq:v-def}
 w=i^n\prod_{j=1}^n Q_jP_j,\qquad
 v_j=\frac1{\sqrt2}w(Q_j+iP_j).
\end{equation}
Then $w=w^*$, $w^2=I$, $wv_jw=-v_j$, and
\begin{equation}\label{eq:v-CAR}
 \{v_j,v_k\}=0,\qquad
 \{v_j,v_k^*\}=2\delta_{jk}I,
 \qquad \tau(v_j^*v_k)=\delta_{jk},\quad
 \tau(v_jv_k)=0.
\end{equation}

By taking real $t_1,\ldots,t_n$ and real positive definite matrices
\begin{equation}\label{eq:strict-Gamma}
 \Gamma_+=(\gamma^+_{jk})_{j,k=1}^n>0,
 \qquad \Gamma_-=(\gamma^-_{jk})_{j,k=1}^n>0,
\end{equation}
we define a linear map $\cL_0$ on $\cM$ as follows:
\begin{equation}\label{eq:L0}
\begin{split}
 \cL_0(x)={}&\frac12\sum_{j,k=1}^n e^{t_j+t_k}(\gamma^+_{jk}+\gamma^-_{jk})v_k^*xv_j
       +\frac12\sum_{j,k=1}^n e^{-t_j-t_k}(\gamma^+_{jk}+\gamma^-_{jk})v_k xv_j^*\\
 &+\frac12\sum_{j,k=1}^n e^{t_j-t_k}(\gamma^-_{jk}-\gamma^+_{jk})v_k xv_j
       +\frac12\sum_{j,k=1}^n e^{t_k-t_j}(\gamma^-_{jk}-\gamma^+_{jk})v_k^*xv_j^*.
\end{split}
\end{equation}
Note that $\cL_0$ is a completely positive map on $\cM$.
Let $\displaystyle \mathbf{y}=\frac12\cL_0(I)$.
The Lindbladian generator $\cL$ is given by
\begin{equation}\label{eq:L-def}
\cL(x)=\mathbf{y}x+x\mathbf{y}-\cL_0(x),\qquad x\in \cM, 
\end{equation}
and the associated semigroup is denoted by $\Phi_t=e^{-t\cL}$.

Let $R_\pm=\Gamma_\pm^{1/2}$ be the square root of $\Gamma_{\pm}$ and 
\begin{equation}\label{eq:jumps}
\begin{split}
 K_{\ell,+}&=\frac1{\sqrt2}\sum_{j=1}^n(R_+)_{\ell j}
                       (e^{t_j}v_j-e^{-t_j}v_j^*),\\
 K_{\ell,-}&=\frac1{\sqrt2}\sum_{j=1}^n(R_-)_{\ell j}
                       (e^{t_j}v_j+e^{-t_j}v_j^*),\qquad 
\end{split}
\end{equation}
for $1\leq\ell\leq n$, where $(X)_{jk}$ denotes the $(j,k)$-entry of a matrix $X$.

\begin{lemma}\label{lem:GKSL}
For any $x\in \cM$, we have that
\begin{equation}\label{eq:GKSL}
\begin{split}
 \cL_0(x)&=\sum_{\ell,\varepsilon}K_{\ell,\varepsilon}^*xK_{\ell,\varepsilon},\\
\cL(x)&=
   \sum_{\ell,\varepsilon} \frac12\{K_{\ell,\varepsilon}^*K_{\ell,\varepsilon},x\}-\sum_{\ell,\varepsilon}
 K_{\ell,\varepsilon}^*x K_{\ell,\varepsilon}
\end{split}
\end{equation}
where $\varepsilon\in \{+, -\}$ and $1\leq \ell\leq n$.
Moreover,
\begin{equation}\label{eq:noise-span}
 \cV:=\spanop_\bC\{K_{\ell,\varepsilon}\}
 =\spanop_\bC\{v_j,v_j^*:1\leq j\leq n\},
 \qquad C^*(\{K_{\ell,\varepsilon}\})=M_d(\bC).
\end{equation}
\end{lemma}
\begin{proof}
Equations \eqref{eq:GKSL} follow from a direct matrix computation of $R_{\pm}^2=\Gamma_{\pm}$.
The invertibility of $\Gamma_{\pm}$ implies that \eqref{eq:noise-span} holds.
\end{proof}

Lemma \ref{lem:GKSL} ensures that the semigroup generated by $\cL$ in \eqref{eq:L-def} is a quantum Markov semigroup.
We shall show that the quantum Markov semigroup with generator $\cL$ described by \eqref{eq:L-def} is primitive.
\begin{lemma}\label{lem:primitive}
We have that $\ker \cL=\bC I$ and there is a unique invariant density $D_\rho>0$ for the semigroup.
Furthermore,
\begin{equation}\label{eq:limit}
 \lim_{t\to\infty}\|\Phi_t-\bE_\Phi\|=0.
\end{equation}
\end{lemma}
\begin{proof}
Suppose that $D_0\geq0$ is a density matrix.
We obtain an invariant density by taking a cluster point of the averages
\[
 \frac1T\int_0^T\Phi_s^*(D_0)\,ds.
\]
Each average is a density matrix, and the set of density matrices is
compact in finite dimension, so a convergent subsequence exists as
$T\to\infty$.
Since
\[
 \cL^*\left(\frac1T\int_0^T\Phi_s^*(D_0)\,ds\right)
 =\frac{D_0-\Phi_T^*(D_0)}{T}\longrightarrow0
 \qquad(T\to\infty),
\]
the continuity of $\cL^*$ implies that every such limit lies in
$\ker\cL^*$ and is therefore invariant.
Let $D_\rho$ be any invariant density and $p$ its support projection.
For $\xi\in\ker D_\rho$, the Schr\"odinger form of
\eqref{eq:GKSL} yields
\[
 0=\ip{\xi}{(-\cL^*)(D_\rho)\xi}
  =\sum_{\ell,\varepsilon}
       \|D_\rho^{1/2}K_{\ell,\varepsilon}^*\xi\|^2.
\]
Thus $(I-p)K_{\ell,\varepsilon}p=0$. By
\eqref{eq:noise-span}, $p\bC^d$ is invariant under every $v_j$ and
$v_j^*$. These operators generate $M_d$, so $p=I$. Every invariant
density is therefore faithful.

A direct expansion gives
\begin{equation}\label{eq:carre}
\Gamma(D, D)= -\cL(D^*D)+\cL(D^*)D+D^*\cL(D)
  =\sum_{\ell,\varepsilon}[K_{\ell,\varepsilon},D]^*
                         [K_{\ell,\varepsilon},D].
\end{equation}
If $\cL(D)=0$, apply $\Tr(D_\rho\,\cdot)$ to
\eqref{eq:carre}. Faithfulness implies that $D$ commutes with all the
noise operators $K_{\ell, \varepsilon}$.
Thus $D$ commutes with all of $M_d(\bC)$, and hence $D\in\bC I$. 
Since $\dim\ker \cL^* =\dim\ker \cL=1$, the invariant density $D_\rho$ is unique.

The same argument excludes nonzero imaginary eigenvalues. Namely, if
$\cL(D)=i\omega D$ with $\omega\in\bR$, the last two terms on the left of
\eqref{eq:carre} cancel. Invariance again implies that all commutators
vanish. Hence $D$ is scalar and $\omega=0$ whenever $D\ne0$.
The maps $\Phi_t$ are contractions in operator norm, so
$\Re\spec \cL\geq0$, and every eigenvalue on the imaginary axis is
semisimple: a nontrivial Jordan block would give polynomial growth of
$\Phi_t$. All other eigenvalues have strictly positive real parts.
The finite-dimensional Jordan decomposition therefore gives norm
convergence to a projection onto $\bC I$. Invariance of $D_\rho$
identifies that projection as the map $\bE_\Phi$ in \eqref{eq:limit}.
\end{proof}

Combining this with Lemma 7.5 in \cite{JWW26}, we conclude that the above quantum Markov semigroup is primitive and bimodule KMS-symmetric. 
Denoting by $D_\rho$ the unique invariant density, we now give a criterion for this quantum Markov semigroup to be KMS-symmetric with respect to $D_\rho$. 

Let $K_1,\ldots,K_{2n}$ denote, in order, the noise operators in \eqref{eq:jumps}. 
They are explicitly determined by $\Gamma_\pm$ and the parameters $t_j$. 
We choose a self-adjoint basis $\{f_j\}_{1\leq j\leq 2n}$ as follows and express the noise operators $\{K_j\}_{1\leq j\leq 2n}$ in this basis:

\begin{equation}\label{eq:majoranas}
    f_j=\frac{v_j+v_j^*}{2},\qquad f_{n+j}=\frac{v_j^*-v_j}{2i},
 \end{equation}
 where $1\leq j\leq n$ and $ \{f_j,f_k\}=\delta_{jk}I$.

\begin{theorem}[KMS Symmetry]\label{thm:kmscriterion}
The following conditions are equivalent:
\begin{enumerate}
 \item $\cL$ is KMS-symmetric with respect to its invariant state with density matrix $D_\rho$;
 \item $\cL_0$ is KMS-symmetric with respect to $D_\rho$: for every $X\in M_d$,
 \begin{equation}\label{eq:kms-cp-condition}
 \cL_0(X)=D_\rho^{-1/2}
     \cL_0^*(D_\rho^{1/2}XD_\rho^{1/2})D_\rho^{-1/2};
 \end{equation}
 \item For any $1\leq j,k\leq n$, 
 \begin{equation}\label{eq:kms-entry-criterion}
 (t_j^2-t_k^2)\gamma^+_{jk}
 =(t_j^2-t_k^2)\gamma^-_{jk}=0.
 \end{equation}
\end{enumerate}
Equivalently,  $\Gamma_\pm$ satisfy
\[
 [\Gamma_\pm,\operatorname{diag}(t_1^2,\ldots,t_n^2)]=0,
\]
that is, they are block diagonal with respect to the partition of the indices
determined by equal values of $|t_j|$.

Moreover, each of the above conditions implies
\begin{equation}\label{eq:y-commutes-sigma}
 [\cL_0(I),D_\rho]=0.
\end{equation}
\end{theorem}

\begin{proof}
We first prove that conditions (1) and (2) are equivalent.
We use $\sharp$ to denote the adjoint with respect to the
KMS inner product induced by $D_\rho$.
A direct computation gives for all $X\in \cM$,
\[
\cL_0^\sharp(X)
=
\sum_{j=1}^{2n} (D_\rho^{-1/2} K_jD_\rho^{1/2}) X( D_\rho^{1/2}K_j^*D_\rho^{-1/2}),
\]
and
\[
\cL^\sharp(X)
=
\bigl(D_\rho^{-1/2}\mathbf{y}D_\rho^{1/2}\bigr)X
+
X\bigl(D_\rho^{1/2}\mathbf{y}D_\rho^{-1/2}\bigr)
-
\cL_0^\sharp(X).
\]

Suppose first that $\cL=\cL^\sharp$.
Since each $K_j$, and hence each $D_\rho^{1/2}K_j^*D_\rho^{-1/2}$, is traceless,
summing over the standard matrix units yields
\[
\begin{aligned}
&d\mathbf{y}+\operatorname{Tr}(\mathbf{y})I
=
\sum_{i,j=1}^{d}\cL(E_{ij})E_{ji}
=
\sum_{i,j=1}^{d}\cL^\sharp(E_{ij})E_{ji}
=
dD_\rho^{-1/2}\mathbf{y}D_\rho^{1/2}
+\operatorname{Tr}(\mathbf{y})I.
\end{aligned}
\]
Consequently, $[\mathbf{y},D_\rho^{1/2}]=0$.
Substituting this into the expression for $\cL^\sharp$
, we obtain
$\cL_0^\sharp=\cL_0$.

Conversely, suppose that $\cL_0=\cL_0^\sharp$.
Then
\begin{align*}
\cL_0^*(D_\rho)
=
D_\rho^{1/2}\cL_0^\sharp(I)D_\rho^{1/2}
=
2D_\rho^{1/2}\mathbf{y}D_\rho^{1/2}.
\end{align*}
Since $D_\rho$ is invariant,
$0
=\cL^*(D_\rho)
=[D_\rho^{1/2},[D_\rho^{1/2},\mathbf{y}]]$.
Thus
\begin{align*}
0
=
\Tr\bigl(
\mathbf{y}[D_\rho^{1/2},[D_\rho^{1/2},\mathbf{y}]]
\bigr)
=
\Tr\bigl(
[D_\rho^{1/2},\mathbf{y}]^*
[D_\rho^{1/2},\mathbf{y}]
\bigr).
\end{align*}
It follows that $[D_\rho^{1/2},\mathbf{y}]=0$.
Substituting this into the expression for $\cL^\sharp$
, we conclude that
$\cL^\sharp=\cL$.

We next prove that conditions (2) and (3) are equivalent.
If $\cL_0=\cL_0^\sharp$, then their Kraus operators span the same linear space; namely, 
\begin{align*}
 \mathcal V:=\operatorname{span}_{\mathbb C}\{K_r\}
 =\operatorname{span}_{\mathbb C}\{D_\rho^{1/2}K_r^*D_\rho^{-1/2}\}.
 \end{align*}
Since
$\mathcal V=\operatorname{span}_{\mathbb C}\{v_j,v_j^*:1 \le  j \le n\}=\mathcal V^*$,
we have
\begin{equation*}
\begin{aligned}
 D_\rho^{1/2}\mathcal V D_\rho^{-1/2}
 =& D_\rho^{1/2}
   \operatorname{span}_{\mathbb C}\{K_r^*\}
   D_\rho^{-1/2} \\
 =&
   \operatorname{span}_{\mathbb C}\{ D_\rho^{1/2} K_r^* D_\rho^{-1/2}\}
   =\mathcal V\\
   =& \operatorname{span}_{\mathbb C}\{f_j\}_{1 \le j \le 2n}.
\end{aligned}
\end{equation*}
Writing $\bm f=(f_1,\ldots,f_{2n})^{\mathsf T}$ for the column vector
of these basis elements, we have
$ D_\rho^{1/2}\bm fD_\rho^{-1/2}=M\bm f$, where $M=(M_{jk})_{j,k=1}^{2n}$ is a matrix in $M_{2n}(\bC)$.
The modular map is a positive definite self-adjoint operator with respect
to the Hilbert--Schmidt inner product.
Since $\displaystyle \Tr(f_af_b)=\frac d2\delta_{ab}$, it follows that $M>0$.
By noting that
\[
\begin{aligned}
\delta_{jk}I
=D_\rho^{1/2}\{f_j,f_k\}D_\rho^{-1/2}=\left\{
\sum_{p=1}^{2n}M_{jp}f_p,
\sum_{q=1}^{2n}M_{kq}f_q
\right\}
=(MM^{\mathsf T})_{jk}I,
\end{aligned}
\]
 we obtain that   $M^{-1}=M^T=\overline{M}$, where $\overline{M}$ is the conjugate of $M$ by taking the complex conjugation of each entry.
 
 We rewrite the Kraus operators in terms of $f_j$ as follows:
 \begin{equation}
 K_j=\sum_{a=1}^{2n}A_{aj}f_a,\qquad  A=(A_{aj})_{a,j=1}^{2n}\in M_{2n}(\bC).
\end{equation}
Then $\cL_0$ can be expressed in this self-adjoint basis as follows:
\begin{equation*}
\cL_0(X)=\sum_{a,b}\left(\sum_j\overline{A_{aj}}A_{bj}\right)f_aXf_b
       =\sum_{a,b}(AA^*)_{ba}f_aXf_b.
\end{equation*}

Let $B=AA^*$.
We have that
\begin{equation}\label{eq:B-coefficients}
 B=2\begin{pmatrix}
 D_s\Gamma_+D_s+D_c\Gamma_-D_c
   &i(D_s\Gamma_+D_c+D_c\Gamma_-D_s)\\
 -i(D_c\Gamma_+D_s+D_s\Gamma_-D_c)
   &D_c\Gamma_+D_c+D_s\Gamma_-D_s
 \end{pmatrix}>0,
\end{equation}
where
\begin{equation*}
    D_s=\diag(\sinh t_1,\ldots,\sinh t_n),\qquad
 D_c=\diag(\cosh t_1,\ldots,\cosh t_n).
\end{equation*}
    A direct computation shows that
\[
 \cL_0^\sharp(X)=\sum_{j,k}B'_{k,j}f_jXf_k,
\]
where the coefficient matrix $B'=(B'_{j,k})$ satisfies
$B'=M^{\mathsf T}\overline{B}\,\overline{M}
   =M^{-1}\overline{B}M^{-1}$.

Since $\cL_0=\cL_0^\sharp$, we obtain
$MBM=\overline B$.
Note that the equation $XBX=\overline{B}$ has a unique positive definite solution $X_0:=B^{-1/2}\bigl(B^{1/2}\overline B B^{1/2}\bigr)^{1/2}B^{-1/2}$.
Hence $M=X_0$.

Observe that
\begin{equation}\label{eq:y-B}
\begin{aligned}
\mathbf{y}
&=\frac14\operatorname{Tr}(B)I-\frac14\sum_{j,k}(B-\overline B)_{jk}f_jf_k,\\
D_\rho^{\frac12}\mathbf{y}D_\rho^{-\frac12}
&=\frac14\operatorname{Tr}(B)I
  -\frac14\sum_{j,k}
  \bigl(X_0^{\mathsf T}(B-\overline B)X_0\bigr)_{jk}f_jf_k.
\end{aligned}
\end{equation}
Since $[D_\rho,\mathbf{y}]=0$, we have
$\mathbf{y}=D_\rho^{\frac12}\mathbf{y}D_\rho^{-\frac12}$.
Moreover, the operators $\{f_jf_k:j<k\}$ are linearly independent,
and the matrix
$X_0^{\mathsf T}(B-\overline B)X_0-(B-\overline B)$
is skew-symmetric. Hence
\[
 X_0^{\mathsf T}(B-\overline B)X_0=B-\overline B.
\]
Using $X_0^{\mathsf T}=X_0^{-1}$, we conclude that
$[B-\overline B,X_0]=0$, which is 
 \begin{equation}\label{eq:kms-Gamma-criterion}
 \left[B-\overline B,\,
 B^{-1/2}
 \bigl(B^{1/2}\overline B B^{1/2}\bigr)^{1/2}
 B^{-1/2}\right]=0.
 \end{equation}

The identity
\[
X_0^{-1}=X_0^{\mathsf T}=\overline{X_0}
\]
holds independently of the KMS-symmetry assumption.
Indeed, conjugating $X_0BX_0=\overline B$ gives
\[
\overline{X_0}\,\overline B\,\overline{X_0}=B,
\qquad
X_0^{-1}\overline B X_0^{-1}=B.
\]
Both $\overline{X_0}$ and $X_0^{-1}$ are positive definite, so uniqueness
of the positive definite solution implies
$\overline{X_0}=X_0^{-1}$.
Since $X_0$ is Hermitian, we also have $X_0^{\mathsf T}=\overline{X_0}$.

Suppose that \eqref{eq:kms-Gamma-criterion} holds.
By $X_0^{-1}=X_0^{\mathsf{T}}=\overline{X_0}$, we see that $\log X_0$ is a self-adjoint, skew-symmetric matrix.
For this part of the proof only, define
\[
 h=-\frac12\sum_{a,b}(\log X_0)_{ab}f_af_b=h^*,
 \qquad \eta=\frac{e^{2h}}{\Tr(e^{2h})}.
\]
The CAR identity
$[f_af_b,f_c]=\delta_{bc}f_a-\delta_{ac}f_b$
yields $[h,\bm f]=(\log X_0)\bm f$, and hence
$\eta^{1/2}\bm f\eta^{-1/2}=X_0\bm f$.
By $X_0 B X_0=\overline{B}$, we see that $\cL_0$ is KMS-symmetric with respect to $\eta$.
Moreover, \eqref{eq:y-B} and \eqref{eq:kms-Gamma-criterion} imply that
$\eta^{1/2}\mathbf{y}\eta^{-1/2}=\mathbf{y}$.
Thus $\cL$ is also KMS-symmetric with respect to $\eta$.
By uniqueness of the invariant state, $\eta=D_\rho$.
Therefore, $\cL$ is KMS-symmetric with respect to $D_\rho$.

Finally, we interpret the condition \eqref{eq:kms-Gamma-criterion} in entrywise form.
Let
\[
 R=\begin{pmatrix}D_c&iD_s\\-iD_s&D_c\end{pmatrix},
 \qquad
 G=\begin{pmatrix}\Gamma_-&0\\0&\Gamma_+\end{pmatrix}.
\]
Then $B=2RGR$.
Since $\overline B=R^{-2}BR^{-2}$ and $R^{-2}>0$,
the uniqueness of the positive definite solution gives
\begin{equation}\label{eq:Theta-explicit}
 X_0=R^{-2}
 =\begin{pmatrix}
 \operatorname{diag}(\cosh 2t_j)&-i\operatorname{diag}(\sinh 2t_j)\\
 i\operatorname{diag}(\sinh 2t_j)&\operatorname{diag}(\cosh 2t_j)
 \end{pmatrix}.
\end{equation}
Multiplying the commutator in \eqref{eq:kms-Gamma-criterion} on the left and right by $R$, we obtain
\[
\begin{aligned}
 R[X_0,B-\overline B]R
 &=R^{-1}(B-\overline B)R
   -R(B-\overline B)R^{-1}\\
 &=2GR^2-2R^{-2}G-2R^2G+2GR^{-2}\\
 &=2[G,R^2+R^{-2}].
\end{aligned}
\]

Since
\[
 R^2+R^{-2}
 =2\begin{pmatrix}
 D_c^2+D_s^2&0\\
 0&D_c^2+D_s^2
 \end{pmatrix}
 =2\begin{pmatrix}
 \operatorname{diag}(\cosh 2t_j)&0\\
 0&\operatorname{diag}(\cosh 2t_j)
 \end{pmatrix},
\]
the commutation condition can be written entrywise as
\[
 (\cosh 2t_j-\cosh 2t_k)\gamma^\pm_{jk}=0,
 \qquad 1\leq j,k\leq n.
\]
This is equivalent to \eqref{eq:kms-entry-criterion}.
\end{proof}

\begin{remark}
The criterion \eqref{eq:kms-entry-criterion} does not involve the invariant
density matrix $D_\rho$ and requires no explicit computation of matrix square roots.
\end{remark}

\begin{remark}
Theorem~\ref{thm:kmscriterion} shows that, if $|t_1|,\ldots,|t_n|$ are pairwise distinct, 
then the generator $\cL$ in this example is KMS-symmetric with respect to $D_\rho$ if and only if both positive definite matrices $\Gamma_\pm$ are diagonal.
By Lemma~\ref{lem:primitive} and the bimodule KMS-symmetry argument above,
the QMS generated by $\cL$ is primitive and bimodule KMS-symmetric,
with a unique invariant state whose density matrix is $D_\rho$.

However, whenever at least one of $\Gamma_+$ and $\Gamma_-$ is non-diagonal, $\cL$ is not KMS-symmetric with respect to $D_\rho$ and hence is not KMS-symmetric with respect to any state. 
Nevertheless, we can prove that the quantum Markov semigroup generated by $\cL$ satisfies CMLSI whenever $\Gamma_\pm>0$.
\end{remark}

When both $\Gamma_+$ and $\Gamma_-$ are diagonal, the invariant density
matrix $D_\rho$ is explicitly given by
\begin{align}\label{eq:drho}
\frac{\rho_0}{\Tr(\rho_0)} (=:D_{\rho_0}),
 \qquad
 \rho_0=\prod_{j=1}^n
 \left(
 \frac12e^{-2t_j}v_j^*v_j
 +\frac12e^{2t_j}v_jv_j^*
 \right).
\end{align}

\begin{theorem}[GNS symmetry]
\label{thm:gns-general-parameters}
The following conditions are equivalent:
\begin{enumerate}
\item[\textup{(i)}]
\(\mathcal L\) is GNS-symmetric with respect to \(D_\rho\).

\item[\textup{(ii)}]
For every \(1\leq j,k\leq n\), including \(j=k\),
\[
(t_j-t_k)(\gamma^+_{jk}+\gamma^-_{jk})=0,
\qquad
(t_j+t_k)(\gamma^-_{jk}-\gamma^+_{jk})=0.
\]

\end{enumerate}

These conditions admit the following block characterization.
For each \(a\in\{|t_1|,\ldots,|t_n|\}\), set
\[
I_a=\{j:|t_j|=a\}.
\]
Both \(\Gamma_+\) and \(\Gamma_-\) must be block diagonal with
respect to the partition \(\{I_a\}_a\). For every \(a>0\), their
corresponding principal blocks satisfy
\[
\Gamma_-^{(a)}=J_a\Gamma_+^{(a)}J_a,
\qquad
J_a=\operatorname{diag}
\bigl(\operatorname{sgn}(t_j)\bigr)_{j\in I_a}.
\]
If \(I_0\neq\varnothing\), the two blocks on \(I_0\) may be chosen
independently, subject to real symmetry and positive definiteness.
\end{theorem}

\begin{proof}
Write
\[
u_{jk}=\gamma^+_{jk}+\gamma^-_{jk},
\qquad
z_{jk}=\gamma^-_{jk}-\gamma^+_{jk}.
\]
We first identify the invariant density whenever the KMS criterion
from Theorem~\ref{thm:kmscriterion} holds:
\begin{equation}
\label{eq:gns-general-kms-criterion}
(t_j^2-t_k^2)\gamma^\pm_{jk}=0
\qquad(1\leq j,k\leq n).
\end{equation}
Set
\[
H=\sum_{j=1}^nt_j\frac12v_j^*v_j,\qquad
D_\eta=\frac{e^{-4H}}{\operatorname{Tr}(e^{-4H})},
\qquad S=D_\eta^{1/2}.
\]
The CAR relations \eqref{eq:v-CAR} yield
\[
[H,v_j]=-t_jv_j,\qquad
[H,v_j^*]=t_jv_j^*,
\]
and therefore
\[
Sv_jS^{-1}=e^{2t_j}v_j,\qquad
Sv_j^*S^{-1}=e^{-2t_j}v_j^*.
\]
Since the coefficients in the noise operators are real,
\[
SK_{\ell,+}^*S^{-1}=-K_{\ell,+},
\qquad
SK_{\ell,-}^*S^{-1}=K_{\ell,-}.
\]
Thus \(\mathcal L_0\) is KMS-symmetric with respect to \(D_\eta\),
and
\[
\mathcal L_0^*(D_\eta)
=S\mathcal L_0(I)S=2S\mathbf{y} S.
\]
Furthermore, a direct expansion of \( \displaystyle \mathbf{y}=\frac12\mathcal L_0(I)\)
gives
\begin{equation}\label{eq:commutor}
\begin{aligned}
\bigl[H, \mathbf{y}\bigr]
={}&\frac12\sum_{j<k}
u_{jk}(t_j-t_k)\sinh(t_j+t_k)
\bigl(v_j^*v_k-v_k^*v_j\bigr)\\
&+\frac12\sum_{j<k}
z_{jk}(t_j+t_k)\sinh(t_j-t_k)
\bigl(v_j^*v_k^*-v_kv_j\bigr).
\end{aligned}
\end{equation}
If \(|t_j|\neq|t_k|\), condition
\eqref{eq:gns-general-kms-criterion} gives \(u_{jk}=z_{jk}=0\).
If \(|t_j|=|t_k|\), both scalar prefactors in the displayed
commutator vanish. 
Hence \([H,\mathbf{y}]=0\), and consequently
\([S,\mathbf{y}]=[D_\eta,\mathbf{y}]=0\). 
It follows that
\[
\mathcal L^*(D_\eta)
=\mathbf{y} D_\eta+D_\eta \mathbf{y}-2S \mathbf{y} S=0.
\]
Uniqueness of the invariant density gives
\[
D_\rho=D_\eta,\qquad [\mathbf{y},D_\rho]=0.
\]
This establishes the claimed density formula under \eqref{eq:gns-general-kms-criterion}.

We now prove \(\textup{(i)}\Rightarrow\textup{(ii)}\).
GNS symmetry implies KMS symmetry, so Theorem~\ref{thm:kmscriterion} yields
\eqref{eq:gns-general-kms-criterion}.
The preceding calculation therefore applies.
For the GNS inner product
\[
\langle X,Y\rangle_{\mathrm{GNS},D_\rho}
=\operatorname{Tr}(D_\rho X^*Y),
\]
the adjoint of a linear map \(\Psi\) is
\[
\Psi^\flat(X)=\Psi^*(XD_\rho)D_\rho^{-1},
\]
where \(\Psi^*\) denotes the Hilbert--Schmidt adjoint.
Since \([\mathbf{y},D_\rho]=0\), we have
\[
\mathcal L^\flat(X)
=\mathbf{y}X+X\mathbf{y}-\mathcal L_0^\flat(X).
\]
Thus GNS symmetry of \(\mathcal L\) is equivalent to
\(\mathcal L_0^\flat=\mathcal L_0\).

Using
\[
D_\rho v_jD_\rho^{-1}=e^{4t_j}v_j,\qquad
D_\rho v_j^*D_\rho^{-1}=e^{-4t_j}v_j^*,
\]
we obtain
\[
\begin{aligned}
\mathcal L_0^\flat(X)
=\frac12\sum_{j,k=1}^n\Bigl(
&u_{jk}e^{3t_j-t_k}v_k^*Xv_j
+u_{jk}e^{-3t_j+t_k}v_kXv_j^*\\
&+z_{jk}e^{3t_j+t_k}v_kXv_j
+z_{jk}e^{-3t_j-t_k}v_k^*Xv_j^*
\Bigr).
\end{aligned}
\]
The operators \(v_1,\ldots,v_n,v_1^*,\ldots,v_n^*\) are linearly
independent. Their left and right multiplication maps are also
linearly independent, as follows from
\[
\operatorname{vec}(AXB)
=(B^T\otimes A)\operatorname{vec}(X).
\]
Comparing the coefficients in \(\mathcal L_0^\flat\) with those
in \(\mathcal L_0\), we therefore obtain
\[
u_{jk}\bigl(e^{2(t_j-t_k)}-1\bigr)=0,
\qquad
z_{jk}\bigl(e^{2(t_j+t_k)}-1\bigr)=0.
\]
Since all parameters are real, these identities are precisely
\textup{(ii)}.

Conversely, assume \textup{(ii)}.
Whenever \(|t_j|\neq|t_k|\), both \(t_j-t_k\) and \(t_j+t_k\)
are nonzero, so \(u_{jk}=z_{jk}=0\), or equivalently
\(\gamma^+_{jk}=\gamma^-_{jk}=0\).
Thus \eqref{eq:gns-general-kms-criterion} holds, and the initial
calculation gives \(D_\rho=D_\eta\) and \([\mathbf{y},D_\rho]=0\).
Condition \textup{(ii)} makes the coefficients in the displayed
formula for \(\mathcal L_0^\flat\) equal to those in
\(\mathcal L_0\). Hence
\(\mathcal L_0^\flat=\mathcal L_0\) and
\(\mathcal L^\flat=\mathcal L\), proving \textup{(i)}.

For the block characterization, entries connecting different
absolute values vanish. Within \(I_a\), \(a>0\), condition
\textup{(ii)} says
\[
\gamma^-_{jk}
=\operatorname{sgn}(t_j)\operatorname{sgn}(t_k)\gamma^+_{jk},
\]
which is exactly
\(\Gamma_-^{(a)}=J_a\Gamma_+^{(a)}J_a\).
On \(I_0\), both scalar factors in \textup{(ii)} vanish.
\end{proof}

\begin{remark}
\label{thm:gns-distinct-parameters}
Suppose that
\[
t_j\neq0\quad(1\leq j\leq n),
\qquad
|t_j|\neq|t_k|\quad(j\neq k).
\]
By Theorem \ref{thm:gns-general-parameters}, we have that \(\mathcal L\) is GNS-symmetric with respect to \(D_\rho\)
if and only if $\Gamma_+=\Gamma_-$ are diagonal.
By Theorem \ref{thm:kmscriterion}, we have that \(\mathcal L\) is KMS-symmetric with respect to \(D_\rho\)
if and only if $\Gamma_+, \Gamma_-$ are diagonal.
\end{remark}

Recall that
\[
D_{\rho_0}
=\frac1Z\prod_{j=1}^n
\bigl(e^{2t_j}(I-\frac12v_j^*v_j)+e^{-2t_j}\frac12v_j^*v_j\bigr),
\qquad
Z=\prod_{j=1}^n2\cosh(2t_j).
\]

\begin{proposition}[Characterization of the explicit invariant density]
\label{prop:candidate-density-distinct-parameters}
Assume that \(t_1,\ldots,t_n\in\mathbb R\) satisfy
\[
|t_j|\neq|t_k|\qquad\text{whenever }j\neq k.
\]
Then
\[
\begin{aligned}
D_\rho=D_{\rho_0}
&\quad\Longleftrightarrow\quad
\mathcal L^*(D_{\rho_0})=0\\
&\quad\Longleftrightarrow\quad
\Gamma_+\text{ and }\Gamma_-\text{ are both diagonal}.
\end{aligned}
\]
\end{proposition}

\begin{proof}
The first equivalence follows from uniqueness of the invariant
density. 

We prove the second equivalence directly.
Let $N_j=\tfrac12v_j^*v_j$ for $1\leq j\leq n$, and set
\[
H=\sum_{j=1}^nt_jN_j,\qquad S=D_{\rho_0}^{1/2}.
\]
By the argument in Theorem \ref{thm:gns-general-parameters}, we have that
\[
\mathcal L^*(D_{\rho_0})=0
\quad\Longleftrightarrow\quad [S,y]=0
\quad\Longleftrightarrow\quad [H,y]=0,
\]
where the last equivalence holds because \(S\) is a positive
scalar multiple of \(e^{-2H}\).

Write
\[
u_{jk}=\gamma^+_{jk}+\gamma^-_{jk},\qquad
z_{jk}=\gamma^-_{jk}-\gamma^+_{jk}.
\]
By Equation \eqref{eq:commutor}, we see that
\([H,\mathbf{y}]=0\) is equivalent to
\[
\begin{aligned}
u_{jk}(t_j-t_k)\sinh(t_j+t_k)&=0,\\
z_{jk}(t_j+t_k)\sinh(t_j-t_k)&=0,
\qquad j<k.
\end{aligned}
\]
By the assumption \(|t_j|\neq|t_k|\), both \(t_j-t_k\)
and \(t_j+t_k\) are nonzero. Since the parameters are real,
the corresponding hyperbolic sine factors are also nonzero.
The preceding identities are therefore equivalent to
\[
u_{jk}=z_{jk}=0
\quad\Longleftrightarrow\quad
\gamma^+_{jk}=\gamma^-_{jk}=0,
\qquad j<k.
\]
Since \(\Gamma_\pm\) are symmetric, this is precisely the
condition that both matrices are diagonal.
\end{proof}


\begin{remark}
In the following, we show that the trace distance of $D_\rho$ and $D_{\rho_0}$ depends on the off-diagonal parts of $\Gamma_{\pm}$.
Fix $t_1=\log 2$ and $t_2=\log 3$, and consider the family
\begin{equation}\label{eq:eta-two-mode-family}
 \Gamma_+(r)=\Gamma_-(r)
 =\frac12
 \begin{pmatrix}
 1&r\\
 r&1
 \end{pmatrix},
 \qquad 0\leq r<1.
\end{equation}
Let $D_\rho(r)$ be the associated invariant density.
The generator is covariant under conjugation by
\(e^{i\theta(N_1+N_2)}\) and has real coefficients in the
occupation number basis. By uniqueness of the invariant density,
\(D_\rho(r)\) has the form
\begin{equation}
\label{eq:two-mode-stationary-density-form}
D_\rho(r)=
\begin{pmatrix}
p_{00}&0&0&0\\
0&p_{01}&x&0\\
0&x&p_{10}&0\\
0&0&0&p_{11}
\end{pmatrix},
\qquad p_{00},p_{01},p_{10},p_{11},x\in\mathbb R,
\end{equation}
where
\begin{equation}
\label{eq:two-mode-stationary-density-entries}
\begin{aligned}
p_{00}(r)&=\frac{299843856-299714256r^2}{Q(r)},\\
p_{01}(r)&=\frac{3701776-2051676r^2}{Q(r)},\\
p_{10}(r)&=\frac{18740241-14865516r^2}{Q(r)},\\
p_{11}(r)&=\frac{231361-101761r^2}{Q(r)},\\
x(r)&=-\frac{2525250r}{Q(r)}
     =-\frac{5250r}{481(1394-1369r^2)},
\end{aligned}
\end{equation}
and $h(r)=1394-1369r^2, Q(r)=231361h(r)$.
In particular,
\[
D_\rho(0)=D_{\rho_0}
=\frac1{1394}\operatorname{diag}(1296,16,81,1).
\]
We measure the difference between the two density matrices using the trace distance:
\[
 \delta(r):=\frac12\|D_\rho(r)-D_{\rho_0}\|_1=\frac{
 875r\left(
 \sqrt{(6475r)^2+16728^2}+17797r
 \right)}
 {1341028(1394-1369r^2)},
 \qquad 0\leq r<1.
\]
In particular, $\delta(0)=0$, whereas $\delta(r)>0$ for every $0<r<1$.

As $r\downarrow0$,
\[
 \delta(r)=\frac{2625}{335257}r+O(r^2).
\]
Thus, in this family of examples, a small off-diagonal coupling produces
a first-order deviation in the density matrix.
On the other hand,
\[
 \lim_{r\uparrow1}\delta(r)
 =
 \frac{1295}{2788}
 +\frac{35\sqrt{321751609}}{1341028}
 \approx0.93264664.
\]
For example, when $r=0.9999$, both coefficient matrices remain strictly
positive definite, yet the trace distance between $D_\rho(r)$ and
$D_{\rho_0}$ already exceeds $0.92$.

Figure~\ref{fig:eta-two-mode-distance} shows the trace distance $\displaystyle \delta(r)=\frac12\|D_\rho(r)-D_{\rho_0}\|_1$ as a function of $r$.

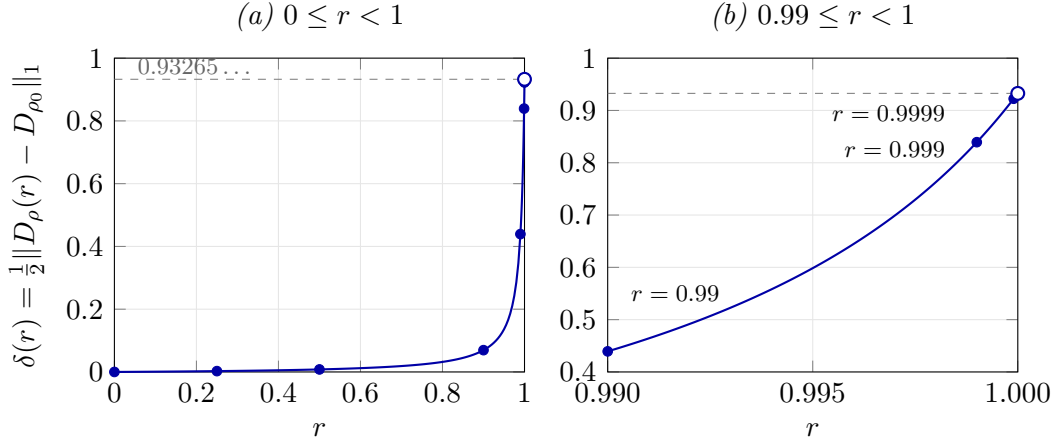
\begin{figure}[htbp]
\centering
\begin{tikzpicture}

\pgfmathdeclarefunction{etadistance}{1}{%
  \pgfmathparse{%
    (2625/335257)*#1
    *(sqrt(1+((6475/16728)*#1)^2)+(17797/16728)*#1)
    /((25/1394)+(1369/1394)*(1-#1)*(1+#1))
  }%
}

\begin{groupplot}[
  group style={
    group size=2 by 1,
    horizontal sep=1.1cm
  },
  width=0.44\linewidth,
  height=0.36\linewidth,
  xlabel={$r$},
  grid=major,
  major grid style={gray!20},
  tick label style={font=\small},
  label style={font=\small},
  title style={font=\small},
  scaled ticks=false,
  clip=false
]

\nextgroupplot[
  title={(a) $0\leq r<1$},
  xmin=0, xmax=1,
  ymin=0, ymax=1,
  xtick={0,0.2,0.4,0.6,0.8,1},
  ytick={0,0.2,0.4,0.6,0.8,1},
  ylabel={$\delta(r)=\frac12\|D_\rho(r)-D_{\rho_0}\|_1$}
]

\addplot[
  blue!65!black, thick,
  domain=0:0.99, samples=200
] {etadistance(x)};

\addplot[
  blue!65!black, thick,
  domain=0.99:1, samples=200
] {etadistance(x)};

\addplot[
  gray, dashed,
  domain=0:1, samples=2
] {0.932646640321};

\addplot[
  only marks, mark=*,
  mark size=1.8pt,
  blue!65!black
] coordinates {
  (0,0)
  (0.25,0.002649880611)
  (0.5,0.008045367997)
  (0.9,0.069475763074)
  (0.99,0.439350085775)
  (0.999,0.839313408556)
  (0.9999,0.922399220676)
};

\addplot[
  only marks, mark=*,
  mark size=2.4pt,
  mark options={
    fill=white,
    draw=blue!65!black,
    thick
  }
] coordinates {(1,0.932646640321)};

\node[
  anchor=west,
  font=\scriptsize,
  text=gray!80!black
] at (axis cs:0.03,0.975) {$0.93265\ldots$};

\nextgroupplot[
  title={(b) $0.99\leq r<1$},
  xmin=0.99, xmax=1,
  ymin=0.4, ymax=1,
  xtick={0.99,0.995,1},
  xticklabels={$0.990$,$0.995$,$1.000$},
  ytick={0.4,0.5,0.6,0.7,0.8,0.9,1}
]

\addplot[
  blue!65!black, thick,
  domain=0.99:1, samples=250
] {etadistance(x)};

\addplot[
  gray, dashed,
  domain=0.99:1, samples=2
] {0.932646640321};

\addplot[
  only marks, mark=*,
  mark size=1.8pt,
  blue!65!black
] coordinates {
  (0.99,0.439350085775)
  (0.999,0.839313408556)
  (0.9999,0.922399220676)
};

\addplot[
  only marks, mark=*,
  mark size=2.4pt,
  mark options={
    fill=white,
    draw=blue!65!black,
    thick
  }
] coordinates {(1,0.932646640321)};

\node[anchor=west,font=\scriptsize]
  at (axis cs:0.9903,0.55) {$r=0.99$};

\node[anchor=east,font=\scriptsize]
  at (axis cs:0.9985,0.825) {$r=0.999$};

\node[anchor=east,font=\scriptsize]
  at (axis cs:0.9985,0.895) {$r=0.9999$};

\end{groupplot}
\end{tikzpicture}

\caption{Trace distance between the candidate density $D_{\rho_0}$
and the invariant density $D_\rho(r)$ for
$t_1=\log 2$, $t_2=\log 3$, and
$\displaystyle \Gamma_+(r)=\Gamma_-(r)
=\frac12\begin{pmatrix}1&r\\r&1\end{pmatrix}$.
The left panel shows the full parameter range,
while the right panel provides an enlarged view of
$0.99\leq r<1$.
Filled markers indicate selected parameter values.
The horizontal dashed lines and open markers indicate
the boundary limit
$\delta(r)\to0.93264664\ldots$ as $r\uparrow1$.}
\label{fig:eta-two-mode-distance}
\end{figure}

\end{remark}

\begin{lemma}\label{lem:modular-closure}
For every $z\in\bC$, we have
\begin{equation}\label{eq:modular-closure}
 D_\rho^z\cV D_\rho^{-z}
   =\cV.
\end{equation}
Recall that $\mathcal V=\spanop_\bC\{v_j, v_j^*: 1\leq j \leq n\}$.
\end{lemma}

\begin{proof}
Let
\begin{align*}
\mathcal Q=&\spanop_\bC\{I,f_af_b:1\leq a,b\leq2n\}, \\
\mathcal Q_{\mathrm{sa}}=&\{h \in \mathcal{Q}:h=h^*\}.
\end{align*}
A direct computation using the CAR relations
$\{f_a,f_b\}=\delta_{ab}I$, stated after \eqref{eq:majoranas}, gives
$[\mathcal Q,\mathcal V]\subseteq\mathcal V
,\quad [\mathcal Q,\mathcal Q]\subseteq\mathcal Q.$

We first show that the logarithm of the invariant density matrix $\log D_\rho$ belongs to $\mathcal Q_{\mathrm{sa}}$.
Define
\[
\omega_t
 :=\Phi_t^*(I/d)
 \geq \Phi_t^*\!\left(\frac{D_\rho}{d\|D_\rho\|}\right)
 =\frac{D_\rho}{d\|D_\rho\|}
 >0,
\]
and set $h_t=\log\omega_t$.
Then $h_t=h_t^*$, and $h_t$ solves the ordinary differential equation
\begin{equation}\label{ode modular action}
\frac{d}{dt} h(t)
 =\K_{e^{h(t)}}^{-1}\bigl(-\cL^*(e^{h(t)})\bigr),
 \qquad h(0)=-(\log d)I.
\end{equation}

We define a function $F$ on $\mathcal Q_{\mathrm{sa}}$ by
\[
 F(h):=\K_{e^h}^{-1}\bigl(-\cL^*(e^h)\bigr),
 \qquad h\in\mathcal Q_{\mathrm{sa}}.
\]
For a fixed $h\in\mathcal Q_{\mathrm{sa}}$, the linear map
\[
 x \longmapsto  \K_{e^h}(x)e^{-h}
 =\int_0^1e^{sh}x e^{-sh}\,ds
\]
is invertible on $M_d(\bC)$.
Moreover, since
\[
 e^hXe^{-h}
 =e^{\operatorname{ad}_h}(X)
 =\sum_{m=0}^{\infty}
 \frac{1}{m!}\operatorname{ad}_h^m(X),
 \qquad \operatorname{ad}_h(X)=[h,X],
\]
this map preserves $\mathcal Q$, and its restriction to $\mathcal Q$
is also invertible.
Note that this map sends $F(h)$ to an element of $\mathcal Q$:
\[
\begin{aligned}
\K_{e^h}(F(h))e^{-h}
 &=-\cL^*(e^h)e^{-h}\\
 &=-\mathbf{y}-e^h\mathbf{y}e^{-h}
   +\sum_{r=1}^{2n}K_r e^hK_r^*e^{-h}
 \in\mathcal Q.
\end{aligned}
\]
It follows that $F(h)\in\mathcal Q$.
Since $\cL^*$ and $\K_{e^h}^{-1}$ preserve self-adjointness,
we also have $F(h)=F(h)^*$, and hence $F(h)\in\mathcal Q_{\mathrm{sa}}$.
This shows that the preceding ODE \eqref{ode modular action} can be solved within the real linear space
$\mathcal Q_{\mathrm{sa}}$.
By the existence and uniqueness theorem for ODEs, we conclude that
$h_t\in\mathcal Q_{\mathrm{sa}}$.

By Lemma~\ref{lem:primitive}, we have
$ h_t=\log\omega_t\longrightarrow\log D_\rho\in\mathcal Q_{\mathrm{sa}}.$
Hence, for every $X\in\mathcal V$ and $z\in\mathbb C$,
\[
\begin{aligned}
 D_\rho^zXD_\rho^{-z}
 =e^{z\operatorname{ad}_{\log D_\rho}}(X)
 =\sum_{m=0}^{\infty}\frac{z^m}{m!}
   \operatorname{ad}_{\log D_\rho}^{\,m}(X)
 \in\mathcal V.
\end{aligned}
\]
Applying the same argument to $-z$ and noting that the two conjugation maps
are inverses of each other, we obtain
\[
 D_\rho^z\mathcal V D_\rho^{-z}=\mathcal V,
 \qquad z\in\mathbb C.
\]
This completes the proof of the lemma.
\end{proof}

\begin{theorem}\label{thm:CMLSI}
The semigroup $\{\Phi_t\}_{t\geq 0}$ with generator $\cL$ defined in \eqref{eq:L-def} satisfies CMLSI with respect to its invariant density matrix $D_\rho$.
Precisely, there exists $\alpha>0$ such that, for every $k\geq1$
and every joint density matrix $D$ on $\bC^d\otimes\bC^k$,
\begin{equation}\label{eq:CMLSI-contraction}
 H\bigl((\Phi_t^*\otimes\Id_k)(D)\Vert D_\rho\otimes D_R\bigr)
 \leq e^{-2\alpha t}H(D\Vert D_\rho\otimes D_R),
 \qquad t\geq0,
\end{equation}
where $D_R=\Tr_{\bC^d}(D)$.
Equivalently, for every positive definite $D$, we have that
\begin{equation}\label{eq:CMLSI-EP}
\mathcal I_L^{(k)}(D):=
 \Tr\left((\cL^*\otimes\Id_k)(D)
       (\log D-\log(D_\rho\otimes D_R))\right)
 \geq2\alpha H(D\Vert D_\rho\otimes D_R).
\end{equation}
Neither conclusion requires $\cL$ to be KMS-symmetric with respect to $D_\rho$.
\end{theorem}

\begin{proof}
We express the quantum Markov semigroup generator $\cL$ as a positive linear combination of two quantum Markov semigroup generators, one of which is GNS-symmetric with respect to $D_\rho$. 
We then apply the CMLSI result for GNS-symmetric quantum Markov semigroups to deduce CMLSI for the original quantum Markov semigroup. 

Let
\begin{equation}\label{eq:quarter-reference}
\begin{aligned}
 & F_j=D_\rho^{1/4}f_jD_\rho^{-1/4},\qquad
 \Psi_{\mathrm{db}}(X)=\sum_{a=1}^{2n}F_a^*XF_a,\\
 & \cL_\mathrm{db}(X)=\frac12\{\Psi_\mathrm{db}(I),X\}-\Psi_\mathrm{db}(X).
 \end{aligned}
\end{equation}
We first show that $\cL_{\mathrm{db}}$ is GNS-symmetric with respect to $D_\rho$.  Since
 $D_\rho^{1/2}F_j^*D_\rho^{-1/2}=F_j$,
the maps $\Psi_\mathrm {db}^\sharp$ and $\Psi_\mathrm{db}$ have the same Kraus operators, where $\Psi_\mathrm {db}^\sharp$ is the KMS adjoint of $\Psi_\mathrm{db}$.
Hence $\Psi_\mathrm{db}$ is KMS-symmetric with respect to $D_\rho$.  
On the other hand, by Lemma~\ref{lem:modular-closure}, for every $t\in\bR$ there exists a real orthogonal matrix $O(t)$ such that 
\begin{align*}
 D_\rho^{it}f_jD_\rho^{-it}=\sum_b O_{jb}(t)f_b.
 \end{align*}
The coefficients are real by self-adjointness of $D_\rho^{it}f_jD_\rho^{-it}$.
Orthogonality follows from the CAR relations
$\{f_j,f_k\}=\delta_{jk}I$, stated after \eqref{eq:majoranas}.
The same transformation formula holds for $F_j$.
Consequently, $\Psi_\mathrm{db}$ commutes with the modular automorphisms $X\mapsto D_\rho^{it}XD_\rho^{-it}$.
In particular, $[\Psi_\mathrm{db}(I),D_\rho]=0$.
Thus $\cL_\mathrm{db}$ is KMS-symmetric and commutes with the modular action, which implies that it is GNS-symmetric with respect to $D_\rho$.

Together with $\cL_\mathrm{db}(I)=0$, this yields
$\cL_\mathrm{db}^*(D_\rho)=0$.
Lemma~\ref{lem:modular-closure} also shows that $\{F_j\}$ and $\{K_r\}$
span the same noise space $\cV$. Hence there exist coefficients $d_{jr}$ such that $\displaystyle  F_j=\sum_r d_{jr}K_r$.
Since the space $\cV$ generates $M_d(\bC)$, the argument in Lemma~\ref{lem:primitive} also applies to $\cL_\mathrm{db}$.
In particular, $\cL_\mathrm{db}$ is primitive.

By the Cauchy--Schwarz inequality, we have for any $\xi\in \bC^d\otimes \bC^k$ and $0 \leq Y\in M_d(\bC)\otimes M_k(\bC)$,
\[
\begin{aligned}
&\left\langle\xi,
  (\Psi_{\mathrm{db}}\otimes\mathrm{id}_k)(Y)\xi
  \right\rangle
 =\sum_j\|Y^{1/2}(F_j\otimes I_k)\xi\|^2\\
&\leq\sum_j
  \left(\sum_r|d_{jr}|^2\right)
  \left(\sum_r\|Y^{1/2}(K_r\otimes I_k)\xi\|^2\right)\\
&=\left(\sum_{j,r}|d_{jr}|^2\right)
  \left\langle\xi,
  (\cL_0\otimes\mathrm{id}_k)(Y)\xi
  \right\rangle.
\end{aligned}
\]
Setting $\displaystyle q=\sum_{j,r}|d_{jr}|^2$, we obtain $\displaystyle \cL_0-\frac1q\Psi_{\mathrm{db}}$ is completely positive.
Define $\displaystyle \cT:=\cL-\frac1q\cL_{\mathrm{db}}$,
then 
\[
 \cT(X)=\frac12\left\{
       \left(\cL_0-\frac1q\Psi_{\mathrm{db}}\right)(I),X
       \right\}
       -\left(\cL_0-\frac1q\Psi_{\mathrm{db}}\right)(X),
\]
so $\cT(I)=0, \; \cT^*(D_\rho)=0$ and $\cT$ is also a quantum Markov semigroup generator.

By Gao and Rouz\'e's CMLSI theorem for finite-dimensional GNS-symmetric quantum Markov semigroups
\cite[Theorem~3.3]{GaoRou22}, there exists $\alpha_\mathrm{db}>0$,
independent of $k$, such that
\begin{equation}\label{eq:reference-CMLSI}
 \cI_{\cL_\mathrm{db}}^{(k)}(D)
 \geq 2\alpha_\mathrm{db}H(D\Vert\ D_\rho\otimes D_R).
\end{equation}
Moreover, the data processing inequality gives
$\cI_\cT^{(k)}(D)\geq0$. Therefore,
\begin{equation}\label{eq:EP-comparison}
 \cI_\cL^{(k)}(D)
 =\frac1q\cI_{\cL_\mathrm{db}}^{(k)}(D)+\cI_\cT^{(k)}(D)
 \geq \frac{2}{q}\alpha_\mathrm{db}H(D\Vert D_\rho\otimes D_R).
\end{equation}
This proves \eqref{eq:CMLSI-EP}.

\end{proof}

\bibliographystyle{abbrv}
\bibliography{kms}

\end{document}